\documentclass[11pt,oneside,reqno]{amsart}
\usepackage[utf8]{inputenc}

\usepackage[top=2cm, bottom=2cm, left = 3cm, right=3cm]{geometry}

\usepackage{amsmath}
\usepackage{amsfonts}
\usepackage{natbib}
\usepackage{MnSymbol}
\usepackage{graphicx}
\usepackage{amsthm}
\usepackage[dvipsnames]{xcolor}
\usepackage{mathtools}
\usepackage{wrapfig}
\usepackage{multicol}
\usepackage{enumerate}
\usepackage[font={footnotesize}]{caption}
\usepackage[backref,colorlinks, citecolor={blue}, linkcolor={blue}]{hyperref}
\usepackage{tikz-cd}

\newcommand{\cl}[1]{\left \lceil #1 \right \rceil}

\theoremstyle{thmstyleone}
\newtheorem{theorem}{Theorem}
\newtheorem{proposition}[theorem]{Proposition}%
\newtheorem{lemma}[theorem]{Lemma}

\theoremstyle{thmstyletwo}%
\newtheorem{remark}{Remark}%

\theoremstyle{thmstylethree}%
\newtheorem{definition}{Definition}%

\author[Rafiqi]{Ahmad Rafiqi$^{1,*}$}

\thanks{2020 {\itshape AMS Mathematics Subject Classification:}  37E30,  05A05, 37B40\\
$^*$Correspondence: arafiqi@aus.edu, ORCID: \href{https://orcid.org/0000-0002-6187-3337}{0000-0002-6187-3337}\\
$^1$Department of Mathematics and Statistics, American University of Sharjah, UAE}

\title{A simple computation of Teichm\"uller Polynomials from Integer Permutations}

\begin{document}
\begin{abstract}
{We present a simple method to compute the Teichm\"uller polynomial of the fibered face of a hyperbolic $3$-manifold $M_\phi$ obtained as the mapping torus of a pseudo-Anosov homeomorphism $\phi$ of a closed surface. We assume $\phi$ has orientable invariant foliations and fixes each singular trajectory. We use a characterisation of such homeomorphisms in terms of a permutation of a finite set of integers to give a direct implementation of McMullens algorithm using train tracks. Train tracks with a single vertex suffice in this case. As an application, for each $p\in\mathbb{Z}_{\geq0}$, we find an infinite sequence of Teichm\"uller polynomials $\Theta_{g,p}$ associated to pseudo-Anosov maps on surfaces of genus $g\geq2$, such that the hyperbolic 3-manifold obtained as the mapping torus has first Betti number $g$. These polynomials realize a positive proportion of bi-Perron units of each degree as pseudo-Anosov stretch-factors.}
\end{abstract}
\maketitle

\section{Introduction}

McMullen \cite{Mc00} defined a polynomial invariant $\Theta_F$, called the \emph{Teichm\"uller polynomial}, of a fibered face $F$ of the unit ball of the Thurston norm on the homology of a fibered hyperbolic $3$-manifold $M$. Evaluating $\Theta_F$ on an integral element $\phi\in\mathbb{R}_+\cdot F$ gives a polynomial whose largest root is the stretch-factor of the monodromy of the fibration associated to $\phi$ \cite[Theorem 4.2]{Mc00}. $\Theta_F$ is defined, up to a unit, as an element of the group ring $\mathbb{Z}[G]$ where $G=H_1(M;\mathbb{Z})/$torsion. McMullen provides a formula for computing it using monodromy invariant train tracks on the fiber.

An algorithm to compute Teichm\"uller polynomials for mapping tori of pseudo-Anosov maps of the punctured disk was given by Lanneau-Valdez \cite{LV17}. In the case of closed surface homeomorphisms a computation was presented by Baik-Wu-Kim-Jo \cite{BW20} for \emph{odd-block} surfaces defined in \cite{BRW16}. More recently, an algorithm to compute these polynomials for closed $3$-manifolds was given by Parlak \cite{P24}, using layered veering triangulations of certain hyperbolic link complements and by obtaining closed $3$-manifolds via Dehn filling.  

We present (Theorem \ref{mainThm}) an elementary formula of McMullen's algorithm in a specific setting: we restrict to mapping tori of pseudo-Anosov homeomorphisms with orientable foliations that have a single singularity, and such that each singular trajectory is fixed. As an application we show that for each $g\geq2$ and each $p\geq0$, there is a pseudo-Anosov homeomorphism $f_{g,p}$ of a closed surface of genus $g$ such that the first Betti number of its mapping torus $M_{g,p}$ equals $g$. The Teichm\"uller polynomial of the associated fibered face equals \begin{equation}\label{polys}
\Theta_{g,p}(t_1,\cdots,t_{g-1},u) = (u-1)^{2g-3}\left(u^2-(\sum_{i=1}^{g-1}t_i+2g+p+1+\sum_{i=1}^{g-1}\frac{1}{t_i})u+1\right),
\end{equation}where $t_1,\cdots,t_{g-1},u$ form a basis for $H_1(M_{g,p};\mathbb{Z})/$torsion, $u$ being the class corresponding to the monodromy $f_{g,p}$.

It is an open problem whether every bi-Perron unit is the stretch-factor of a pseudo-Anosov map. The polynomials $\Theta_{g,p}$ help realize a positive proportion among bi-Perron units in every degree as pseudo-Anosov stretch-factors, (see $\S$\ref{pAmaps} for the relevant definitions). 

\begin{proposition}\label{Prop1}
Let $m\geq2$ and $a_1,\cdots,a_{m-1}$ be non-negative integers such that \[v=(\underbrace{1,\cdots,1}_{a_{m-1}},\underbrace{2,\cdots,2}_{a_{m-2}},\,\cdots\,,\underbrace{m-1,\cdots,m-1}_{a_1},m)\] is a primitive integer vector. Then, for each $a_m\geq3+2(a_1+\cdots+a_{m-1})$, the largest real root of \begin{equation*}
x^{2m}-a_1x^{2m-1}-\cdots-a_{m-1}x^{m+1}-a_mx^m-a_{m-1}x^{m-1}-\cdots-a_1x=0\end{equation*} is a bi-Perron unit and is the stretch-factor of a pseudo-Anosov map on a connected surface. 
\end{proposition} 

When $m=2$, the proportion realized by Proposition \ref{Prop1} as pseudo-Anosov stretch-factors is shown in Figure \ref{degree4fig}. The primitivity condition is rarely violated; for instance, it always holds whenever $a_1$ or $a_{m-1}$ isn't $0$, or when $m$ is prime.

\begin{figure}[h!]
(\includegraphics[width=.6\textwidth]{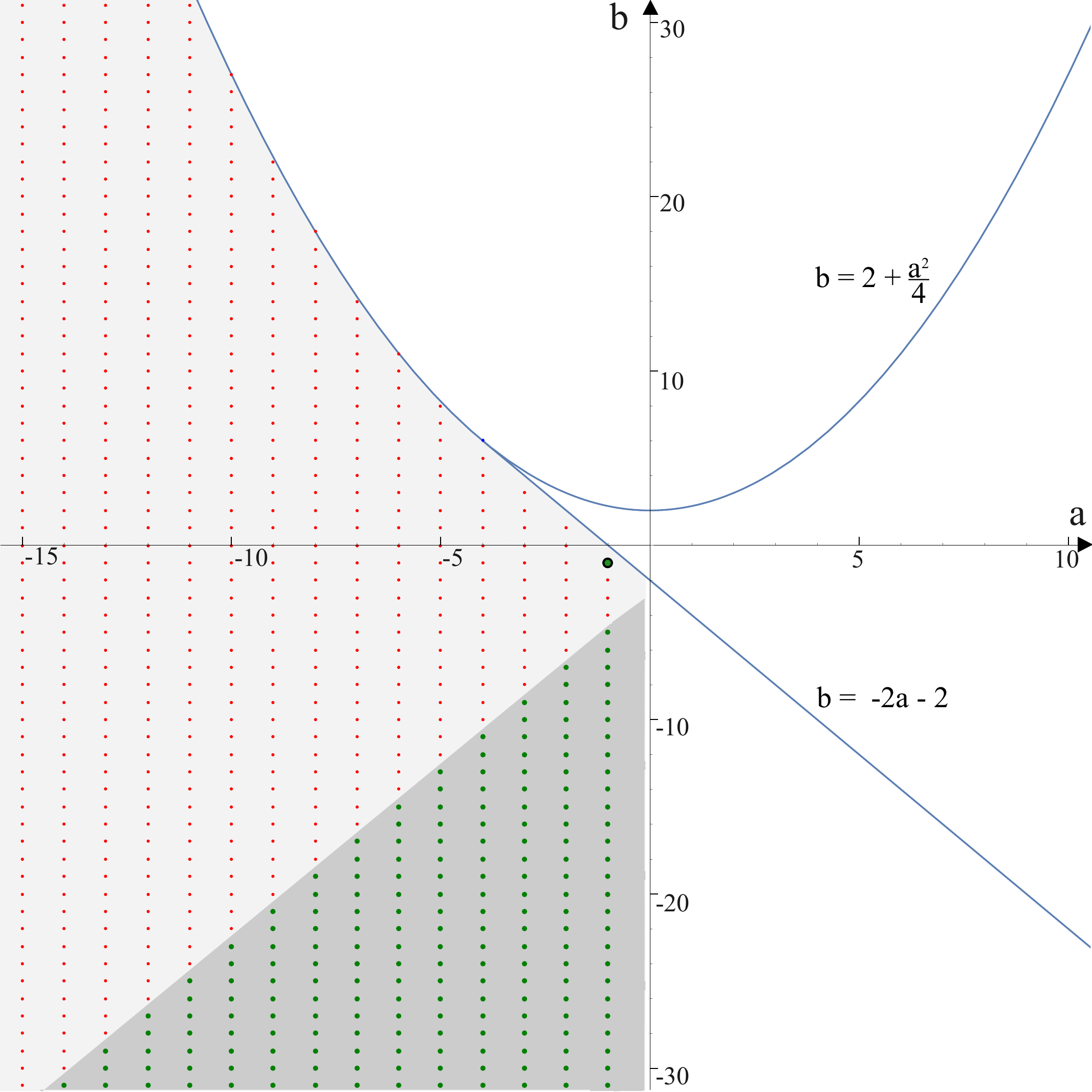}
\caption{$x^4+ax^3+bx^2+ax+1=0$ defines a bi-Perron unit iff $(a,b)\in\mathbb{Z}^2$ belongs to the shaded areas shown. Among these, those with $b\leq2a-3$ are shown to be realized as pseudo-Anosov stretch-factors on connected surfaces by Proposition \ref{Prop1}.}
\label{degree4fig}
\end{figure}

For our algorithm, we use a characterisation of pseudo-Anosov maps in terms of permutations of integers, called \emph{Ordered Block Permutations (OBPs)}, defined in \cite{HRS19}. Briefly, a pseudo-Anosov map with orientable invariant foliations on an oriented surface is \emph{oriented-fixed} if it fixes each singular trajectory. For an oriented-fixed map $f$ on a surface of genus $g$ with $\nu$ distinct singularities, the choice of a segment of the contracting foliation terminating at a singularity decomposes the surface into $n=2g+\nu-1$ rectangles, called \emph{zippered rectangles} ($\S$\ref{zipp}). The rectangles form a Markov partition for $f$ with an $n\times n$ incidence matrix $A$, whose $i^{th}$ column corresponds to the image of the $i^{th}$ rectangle $R_i$. More precisely, the image $f(R_i)$ can be recorded as an ordered list $O_i = (R_{i_1},R_{i_2},\cdots,R_{i_{m_i}})$ according to which rectangles the image of $R_i$ crosses. The OBP records this data in terms of a permutation which records the gluing data of the rectangles, and a vector of positive integers that counts how many image rectangles pass through each rectangle ($\S$\ref{obps}). As is shown in \cite{HRS19}, this permutation of integers suffices to reconstruct the surface and the oriented-fixed map. OBPs are like Interval Exchange Transformations \cite{V82}, except that a characterization in terms of integers lends itself readily to finding the associated algebraic integers.

In the current paper, we focus on maps with a single singularity, so $n=2g$, (see Remark \ref{rmk1}). We form an $n\times n$ matrix $A(\mathbf{t})$ whose entries are Laurent polynomials in variables $\mathbf{t}=(t_1,\cdots,t_{b-1})$, where $b$ is the first Betti number of the mapping torus of $f$. For any $\textbf{c}\in\mathbb{Z}^{b-1}$, let $\mathbf{t}^\mathbf{c}$ denote the product $t_1^{c_1}\cdots t_{b-1}^{c_{b-1}}$. Briefly, $A(\mathbf{t})$ is computed as follows.

\begin{itemize}
\item[$\diamond$] Let $A(\mathbf{t})=\mathbf{0_{n\times n}}$. 
\item[$\diamond$] Choose a basis $\{a^1,\cdots,a^{b-1}\}\subset\mathbb{Z}^n$ for the null-space of $(A^\top-I_{n\times n})$.
\item[$\diamond$] Let $\mathbf{a}_i\in\mathbb{Z}^{b-1}$ denote the vector formed by the $i^{th}$ components of $a^1,\cdots,a^{b-1}$. 
\item[$\diamond$] For each $i\leq n$, and each $R_{i_j}$ in the orbit $O_i$, add to the $\{i_j,i\}^{th}$ entry of $A(\mathbf{t})$ the term \[\displaystyle\prod_{r=1}^{j-1}\mathbf{t}^{{\mathbf{a}}_{i_r}}\] 
\end{itemize}   We show in $\S$\ref{algo} below,  


\begin{theorem}\label{mainThm}
The Teichm\"uller polynomial $\Theta_F$ of the mapping torus $M_f$ associated to the fibered face $F$ determined by $[S]$ is given by \[\Theta_F(t_1,\cdots,t_{b-1},u) = \frac{\det(uI-A(\mathbf{t}))}{u-1}\]
\end{theorem}

In order to deduce information about the stretch-factors associated to points other than $[S]$, it is important to know the boundaries of the fibered cone. Following \cite[\textsection6]{Mc00}, we use the Newton polygon of $\Theta_F$ to determine the fibered cone. We also show that factors $(u-1)^j$ in $\Theta_{g,p}$ don't affect the fibered cone, even though they change the Newton polygon. 

As is shown in \cite{BRW19}, there aren't enough pseudo-Anosov maps with orientable foliations in any genus $g\geq6$ to realize all bi-Perron units of degree $2g$ as stretch-factors. For $g<6$ this is not known; perhaps all degree $4$ bi-Perron units are realized on a genus $2$ surface, experiments do suggest this. In general though, if one is only interested in the set of all stretch-factors one has to not restrict the genus, as we have done here. The numbers realized above are on surfaces of increasing genera.

\subsection*{Reader's guide} In $\S$\ref{background} we review Thurston's theory of pseudo-Anosov maps, fibered $3$-manifolds and the Thurston norm, Teichm\"uller polynomials, and the ordered block permutations that we will use for our algorithm. In $\S$\ref{algo} we describe our algorithm and prove Theorem \ref{mainThm}. In $\S$\ref{deg4}, we will apply the algorithm to a specific family of maps to obtain the sequences of Teichm\"uller polynomials $\Theta_{g,p}$, given in equation (\ref{polys}), and prove Proposition \ref{Prop1}.


\section{Background}\label{background}

\subsection{Pseudo-Anosov Maps}\label{pAmaps} The homotopy classes of homeomorphisms of compact topological surfaces were classified by Thurston \cite{Th88} into three types: periodic, reducible, and pseudo-Anosov. The latter of these contains a unique representative, a \textit{pseudo-Anosov} map $f$ of the surface $S$, which is specified by a finite set $X$ of points on $S$ called the \emph{singularities}, a transverse pair of foliations on $S\setminus X$ invariant under $f$, and a real number $\lambda>1$, called the \emph{stretch-factor}. Each foliation possesses a measure transverse to its leaves. The two foliations are respectively stretched and shrunk by $\lambda$ under $f$, and are hence called the \emph{expanding} and \emph{contracting} foliations of $f$. A finite number $q\geq1$ of trajectories of the foliation emanate from each $X_i\in X$; we will call these \emph{singular trajectories}, and $X_i$ a \emph{$q-$pronged} singularity. For details on this theory, see \cite{FLP}. 

The stretch-factor $\lambda$ is an important topological invariant of the mapping class, $\log(\lambda)$ being the topological entropy of $f$, 
as well as the length of the geodesic corresponding to the mapping class of $f$ in the moduli space of conformal structures on $S$ under the Teichm\"uller metric. 
Fried \cite{Fr85} showed that $\lambda$ is a unit of the ring of integers $O_\lambda$ of the number field $\mathbb{Q}(\lambda)$, and that it is \emph{bi-Perron}: that is, its Galois conjugates $\mu_i$ satisfy $\{1/\lambda\leq |\mu_i|\leq\lambda\}$ with at most one conjugate on each boundary component. 

It is an open problem to determine whether being a bi-Perron unit characterizes pseudo-Anosov stretch-factors. A lot of work has been done towards resolving this question, for instance by \cite{LeStr20,Str17,ShStr16,FLM09,Hir10,HK06,KT08}, among others.

In what follows, we will assume $S$ to be a closed oriented surface and only consider orientation-preserving pseudo-Anosov maps whose invariant foliations are orientable. Orientability of the foliations is equivalent to the two foliations being the integral curves of the real and imaginary parts of a holomorphic $1$-form on a Riemann surface structure on $S$ \cite{HM79}; such a $1$-form is also called an \emph{Abelian differential}. It implies - but is not implied by - having an even number of prongs at each singularity; $q=4,6,8,...$.

\subsection{Fibered $3$-Manifolds and the Thurston Norm} 
Naturally associated to the homeomorphism $f:S\to S$ is a $3$-manifold $M_f$, called the \emph{mapping torus}, defined as \begin{equation*}
M_f = \frac{S\times[0,1]}{<(x,1)\sim(f(x),0),\,\,\,\,\forall x\in S>}
\end{equation*} If a manifold can be obtained as a mapping torus as above, it is called \emph{fibered}, since it is the total space of the fibration $S\hookrightarrow{} M_f\xrightarrow{\pi} S^1$ defined by $\pi(x,t)=t$. When $S$ has genus $g\geq2$, $M_f$ is a hyperbolic $3$-manifold if and only if $f$ is pseudo-Anosov \cite[Proposition 2.6]{Th86H}. 

The oriented surface $S$ can be seen as a representative of a class $[S]\in H_2(M;\mathbb{Z})$, which under Poincar\'e duality $H_2(M;\mathbb{Z})\cong H^1(M;\mathbb{Z})$, is associated to the class of the nowhere-zero $1$-form $\pi^*(dt)$. Nowhere-zero $1$-forms form an open cone in $H^1(M;\mathbb{R})$, since being non-singular is an open condition and invariant under scaling.  Thurston \cite{Th86} defined for any 3-manifold $M$ a semi-norm $||\cdot||_T$ on $H_2(M;\mathbb{R})$ by first defining it for integral classes, then extending it by linearity and continuity to rational and real classes respectively. For $[S]\in H_2(M;\mathbb{Z})$ it is defined as \begin{equation*}
||\,[S]\,||_T := \min\{\chi_-(S') : [S']=[S]\}
\end{equation*}  Here $\chi_-$ computes the negative Euler characteristic of the surface $S'$ ignoring any sphere components. It turns out that $||\cdot||_T$ is a norm when $M$ is irreducible and atoroidal. For a nice detailed exposition of these ideas see \cite[\textsection5.2]{C07}.

As Thurston explains, $||\cdot||_T$ is unlike any norm defined using an inner product in that its unit ball is a finite sided polyhedron rather than ellipsoidal - he shows that this is necessary for a norm that takes integer values on an integer lattice. What's more, within the open cone $\mathcal{C}=\mathbb{R}_+\cdot F\subset H_2(M;\mathbb{R})$ over a top dimensional face $F$ of its unit polyhedron, if a single integral class $[S]\in\mathcal{C}$ corresponds to a fibration of $M$ over the circle, then every integral class in $\mathcal{C}$ corresponds to a fibration of $M$ over the circle. Hence, such a $\mathcal{C}$ is called a \emph{fibered cone}, and $F$ a \emph{fibered face}.

\subsection{Teichm\"uller Polynomials} McMullen \cite{Mc00} defined for each fibered face $F$ of a hyperbolic $3$-manifold $M$ a multivariate polynomial invariant $\Theta_F$ as an element of the group ring $\mathbb{Z}[G]$ where $G=H_1(M;\mathbb{Z})/$torsion. When $\Theta_F$ is evaluated at an integral element $\omega$ of the fibered cone $\mathbb{R}_+\cdot F$, one obtains a polynomial whose largest positive root is the stretch-factor of the monodromy of the fibration associated to $\omega$. As such, it is a powerful tool for finding stretch-factors and has been extensively used to look for them, for instance to look for the minimum ones by \cite{Hir10,HK06,KT08}.

When $M=M_f$ is the mapping torus of a pseudo-Anosov map $f$ of a closed surface $S$, the nowhere-zero element $[S]\in H^1(M_f;\mathbb{Z})$, being in the interior of a fibered cone, uniquely determines a fibered face $F$. We will describe McMullen's algorithm \cite[\textsection3]{Mc00} to compute $\Theta_F$ for this unique fibered face, albeit in our restricted setting, assuming the first Betti number $b=b_1(M_f)$ is $\geq2$, and that $f$ has orientable foliations. 

One first finds on $S$ a minimal trivalent $f$-invariant train track $\tau$ carrying the expanding foliation of the pseudo-Anosov map $f.$ A \emph{train track} is a connected finite graph embedded in the surface $\tau\hookrightarrow S$ such that; at each vertex $v$ of $\tau$, the edges incident at $v$ are tangent to each other; there are at least $3$ edges incident to each vertex; and the complement $S\setminus\tau$ consists of polygons with at least one cusp each. The tangency condition implies that the edges at each vertex can be partitioned into an incoming and outgoing set. $\tau$ is called \emph{$f$-invariant} if there is a homotopy between $f(\tau)\hookrightarrow S$ and $\tau\hookrightarrow S$ that permutes vertices, and sends edges to edge-paths, (see \cite{PP87} for nuances and details). $\tau$ is \emph{trivalent} if exactly $3$ edges meet at each vertex. Furthermore, $\tau$ is \emph{minimal} if it has the minimal number of edges among trivalent $f$-invariant train tracks. 

Let $E$ and $V$ be the set of edges and vertices of $\tau$ respectively. Up to homotopy, $f$ permutes $V$ while each edge $e\in E$ maps to an edge-path consisting of elements of $E$. The action of $f$ on $V$ and $E$ can thus be encoded in terms of non-negative matrices $P_V$ and $P_E$. The largest root of the Perron-Frobenius matrix $P_E$ is the stretch-factor of $f$. 

One can compute, for instance by using the Mayer-Vietoris Sequence, that \begin{equation*}
\mathbb{Z}^b\cong H^1(M_f;\mathbb{Z})\cong (H^1(S;\mathbb{Z}))^f\oplus\mathbb{Z},
\end{equation*} 
where $(H^1(S;\mathbb{Z}))^f$ is the subspace of $H^1(S;\mathbb{Z})$ consisting of $f$-invariant classes; $\omega$ such that $f^*(\omega)=\omega$. Denote by $H$ the dual to this $f$-invariant cohomology of $S$, namely \begin{equation*}
H=\text{Hom}(H^1(S;\mathbb{Z})^f;\mathbb{Z})\cong\mathbb{Z}^{b-1}.
\end{equation*} 

By evaluating elements of $\pi_1(S)$ on $f$-invariant cohomology classes, we obtain a map $\pi_1(S)\to H$ with kernel $K$. Let $p:\widetilde{S}\to S$ be the Galois covering space of $S$ corresponding to the subgroup $K\triangleleft\pi_1(S)$. The deck group of $p$ is naturally $H$. 

Furthermore, since $S\subset M$, we have  \begin{equation*}
H\subset G =  H_1(M;\mathbb{Z})/\text{torsion} = \text{Hom}(H^1(M;\mathbb{Z});\mathbb{Z}). 
\end{equation*}

Since elements $[\gamma]\in p_*(\pi_1(\widetilde{S}))$ correspond to loops in $S$ that evaluate to $0$ on all $\omega\in(H^1(S;\mathbb{Z}))^f$, and $\omega(f_*([\gamma]))=(f^*\omega)([\gamma])=\omega([\gamma])=0$, we have $(f\circ p)_*(\pi_1(\widetilde{S}))\subseteq p_*(\pi_1(\widetilde{S}))$. Thus, we can lift $f$ to a map $\widetilde{f}:\widetilde{S}\to\widetilde{S}$, and we choose one such lift. 

The train track $\tau$ also lifts under $p$ to a train track $\widetilde{\tau}$ on $\widetilde{S}$. The vertices and edges of $\widetilde{\tau}$ can be identified with $H\times V$ and $H\times E$. In fact, if $\mathbf{t}=(t_1,\cdots,t_{b-1})$ is an integral basis for $H$, written multiplicatively, each edge $\widetilde{e}$ of $\widetilde{\tau}$ with $p(\widetilde{e})=e$ can be uniquely labelled as $\mathbf{t}^\mathbf{c}e = t_1^{c_1}\cdots t_{b-1}^{c_{b-1}}e$ for some $\mathbf{c}\in\mathbb{Z}^{b-1}$, and similarly for the vertices.

The edges and vertices of $\widetilde{\tau}$ thus define finite-rank $\mathbb{Z}[H]$-modules, and the action of $\widetilde{f}$ on them can be written as matrices $P_V(\mathbf{t})$ and $P_E(\mathbf{t})$ of Laurent polynomials in the $t_i$. McMullen's algorithm is to then compute the Teichm\"uller polynomial $\Theta_F(\mathbf{t},u)\in\mathbb{Z}[G]=\mathbb{Z}[H]\bigoplus\mathbb{Z}[u]=\mathbb{Z}[t_1^{\pm1},\cdots,t_{b-1}^{\pm1},u^{\pm1}]$ as the quotient of the characteristic polynomials of these matrices. Here $u$ corresponds to $[\widetilde{f}]$.
\begin{theorem}\label{McThm} (\cite[Theorem 3.6]{Mc00}) The Teichm\"uller polynomial of the fibered face $F$ is given by\begin{equation*}
 \Theta_F(\mathbf{t},u) = \frac{\det(uI-P_E(\mathbf{t}))}{\det(uI-P_V(\mathbf{t}))}
 \end{equation*}
\end{theorem}

\subsection{Zippered Rectangles}\label{zipp} The two invariant foliations of a pseudo-Anosov map provide a well-known decomposition of the surface into rectangles that form a Markov partition, called \emph{zippered rectangles} \cite[Proposition 5.3.4]{H16}. We will describe the decomposition for pseudo-Anosov maps with orientable foliations and $\nu$ distinct singularities. Starting at a singularity $X_0$, draw a segment $J$ of the contracting foliation of some positive length, and draw all singular expanding leaves until they intersect $J$. Shorten $J$ till the intersection point furthest from $X_0$, and continue drawing the final singular expanding segment past $J$ till it intersects $J$ again. The complement of the curves thus drawn is a collection of $n=2g+\nu-1$ rectangles, where $g$ is the genus of the surface and $\nu$ is the number of distinct singularities; for details see \cite{HRS19}.

Placing the segment $J$ vertically in the plane $\mathbb{C}$, one can lay down these rectangles in the plane (see Fig.\,\ref{example} below). One may also assume, after possibly reversing the orientation of the expanding foliation, that the horizontal segment that does not contain a singularity is to the right of $J$. The rectangles can be numbered $R_1,\cdots,R_n$ as they are attached to the right of $J$. The order in which they appear on the left of $J$ defines a permutation $\sigma$ of $\mathbb{N}_n=\{1,\cdots,n\}$. The permutation sigma determines the types of singularities, as well as the intersection form on the surface.

%

\subsection{Ordered Block Permutations}\label{obps}
Let us now further restrict our attention to \emph{oriented-fixed} pseudo-Anosov maps $f$, those that fix every singular trajectory. Namely, each prong at each singularity maps to itself. Assume $S$ has genus $g$ and the foliations of $f$ have $\nu$ distinct singularities $X_1,\cdots,X_\nu$. We choose a singular contracting segment $J$ and obtain a zippered rectangle decomposition of $S$ as in section \ref{zipp}. So $n=2g+\nu-1$ rectangles $R_1,\cdots,R_n$ are glued in order along their left vertical edges to the fixed vertical segment $J$. The right edge of $R_i$ is glued in position $\sigma(i)$ on the left of $J$.

Since each $R_i$ has sides alternately on the expanding and contracting foliations, the pseudo-Anosov map shrinks $R_i$ vertically by the stretch-factor $\lambda$, stretches it horizontally by $\lambda$. $f(R_i)$ is thus a thinner but longer rectangle that passes some number of times through each $R_j$. Let $k_i$ be the number of times image rectangles cross $R_i$ and set $\mathbf{k}=\{k_1,\cdots,k_n\}$. 

The pair $(\sigma,\mathbf{k})$ satisfies a combinatorial condition called admissibility that we describe below. As is shown in \cite[Theorem 6.1]{HRS19}, any admissible pair $(\sigma,\mathbf{k})$ can then be used to uniquely construct an oriented surface and an oriented-fixed pseudo-Anosov map of it. In this way a pseudo-Anosov map with orientable foliations on a closed surface of genus $g$ with $\nu$ singularities, which fixes the singular trajectories, can be encoded as a permutation of $n=2g+\nu-1$ positive integers. Every Abelian differential that is invariant under a pseudo-Anosov map $f$ can be constructed this way, even if $f$ doesn't fix the singular trajectories, or isn't orientation preserving \cite[Remark 6.2]{HRS19}. 

\begin{itemize}
\item Define \emph{blocks} $B_1,\cdots,B_n$, where \\$B_1=\{1,...,k_1\}$, $B_2=\{k_1+1,...,k_1+k_2\}$, $\cdots, B_n=\{k_1+...+k_{n-1}+1,...,k_1+...+k_n\}$.\\ The $B_i$'s form a partition of the set $\mathbb{N}_K=\{1,2,\cdots,K=\sum_{i=1}^nk_i\}$.

\item Define the \emph{block function} $\beta:\mathbb{N}_K\to\mathbb{N}_n$ by $\beta(j)=i$ iff $j$ belongs to the block $B_i$. 

\item Finally, define a permutation $\xi$ of the bigger set $\mathbb{N}_K$ by permuting the blocks $B_1,\cdots,B_n$ according to $\sigma$. 
\end{itemize}

That is, define $\xi=\xi_{(\sigma,\mathbf{k})}:\mathbb{N}_K\to\mathbb{N}_K$, called the \emph{ordered block permutation (OBP)} of $(\sigma, \mathbf{k})$, by \begin{equation}
\label{OBPEquation}
\xi(j)\,\, := \sum_{1\le i<\sigma(\beta(j))}k_{\sigma^{-1}(i)} +j - \sum_{1\le i<\beta(j)}k_i.
\end{equation}

The OBP $\xi$ allows one to define another partition of the set $\mathbb{N}_K$ according to the orbits of the first $n$ elements until their first return to $\mathbb{N}_n\subset\mathbb{N}_K$. For each $i\leq n$, let $O_i$ be the ordered set $O_i := (i,\xi(i),\cdots,\xi^{\circ(m_i-1)}(i))$, where $m_i>1$ is the smallest integer such that $\xi^{\circ\,m_i}(i)\leq n$. We also define the first return map $\xi':\mathbb{N}_n\to\mathbb{N}_n$ by setting $\xi'(i)=\xi^{\circ\,m_i}(i)$.

\begin{definition}\cite[Def. 4.1]{HRS19}
An OBP $\xi_{(\sigma,\mathbf{k})}$ is called \textbf{admissible} if
\begin{enumerate}\label{admissibledef}
\item[\textsc{(i)}]
The first return $\xi'$ equals $\sigma$;
\item[\textsc{(ii)}]
$\displaystyle\bigcup_{i=1}^n O_i= \mathbb{N}_K$;
\item[\textsc{(iii)}] Each orbit $O_i$ includes the first and last element of block $B_i$, except $O_{\sigma^{-1}(n)}$ contains the last element $K\in B_n$;
\item[\textsc{(iv)}] The matrix $A$ defined by $A_{ij} = |B_i\cap O_j|$ is irreducible.
\end{enumerate}
\end{definition}

Given an admissible OBP $\xi_{(\sigma,\mathbf{k})}$, the entry $A_{i\,j} = |B_i\cap O_j|$ is the number of times $R_j$ crosses $R_i$. Note that $m_i$ is the sum of the entries of the $i^{th}$ column of $A$, whereas $k_i$ is the sum of the $i^{th}$ row. $A$ has leading eigenvalue equal to the stretch-factor $\lambda$ of $f$ \cite[Proposition 2.2]{HRS19}. In fact, the widths and heights of the rectangles $R_i$ form $\lambda$-eigenvectors of $A^\top$ and $A$ respectively. $A$ represents the induced action $f_*$ on the homology group $H_1(S;\mathbb{Z})$ in terms of a spanning set, which can be identified with the $R_i$. When the number of singularities $\nu=1$, the spanning set is a basis, which is our setting in what follows. 

\section{The Algorithm}\label{algo}
We will describe our algorithm for computing the Teichm\"uller polynomial of the fibered face of the mapping torus $M_f$ of an oriented-fixed pseudo-Anosov homeomorphism $f$ of a closed surface $S=S_g$ with one singularity $P$. Since $f$ is oriented-fixed, choosing a singular contracting segment, we can decompose $S$ into $n=2g$ zippered rectangles and describe $f$ in terms of an ordered block permutation $\xi_{(\sigma,\mathbf{k})}$ as in section \ref{obps} above. 

Each rectangle $R_1,...,R_{n-1}$ has the singularity $P$ on its top and bottom edge, while $R_n$ has the singularity only on its top edge. Connecting the top and bottom singularity of each $R_i$ by an edge $e_i$ oriented upwards, ($e_n$ connecting the singularity on $R_n$ to the singularity on the bottom edge of $R_{\sigma^{-1}(n)}$), we get a CW structure on the surface with a single 2-cell, $n$ $1$-cells $e_1,...,e_n$ and one vertex $P$, see Fig. \ref{example}.

\begin{figure}[h!]
\includegraphics[scale=.42]{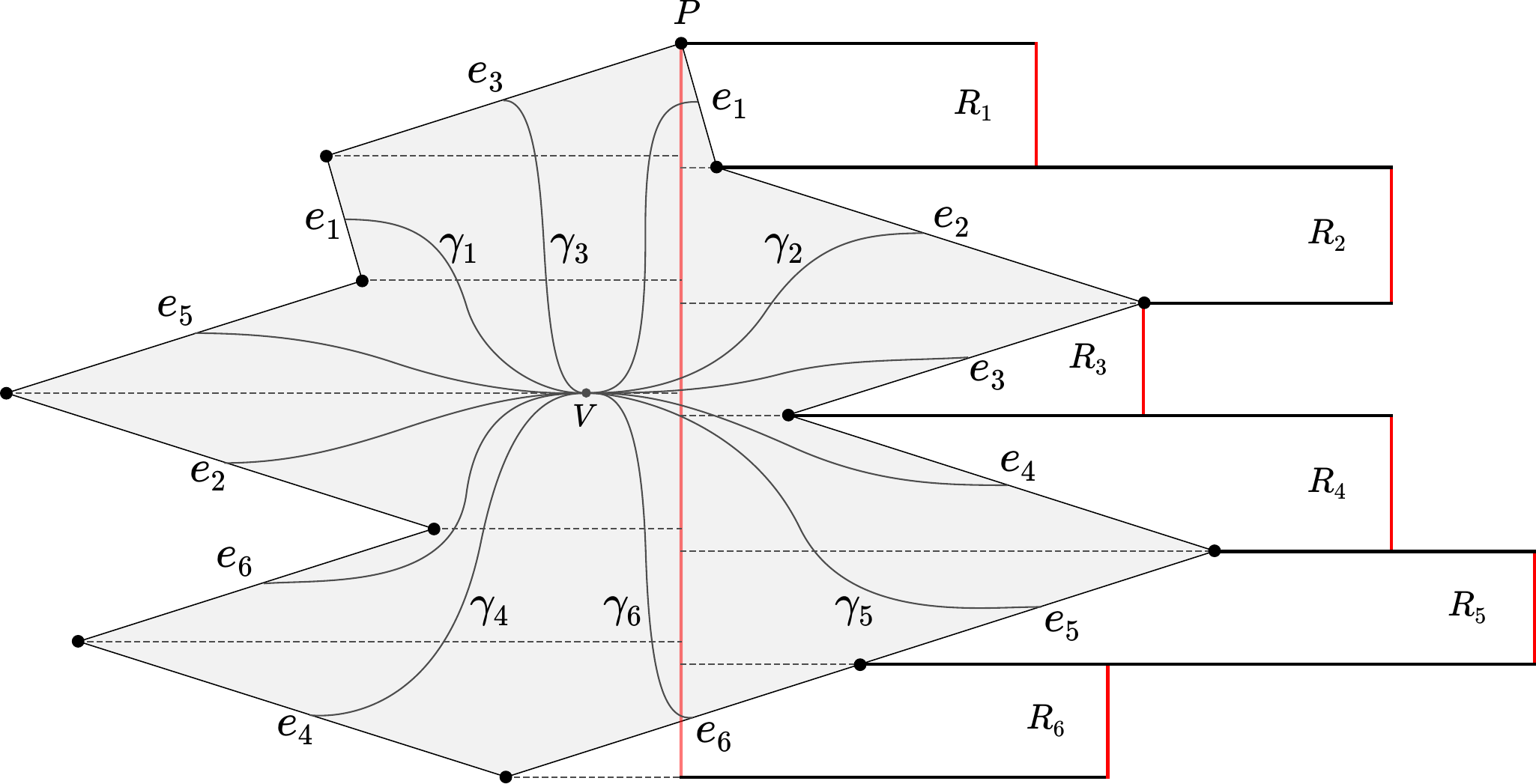}
\caption{An example, with $n=2g=6$, $\sigma=(2,4,1,6,3,5)$ and $\mathbf{k}=(11,12,10,12,10,10)$. We will compute $\Theta_F$ ($\S$\ref{example3/3}) in this case to be $(u-1)^3(u^2-(t_1+t_1t_2+7+1/t_1+1/{t_1t_2})u+1)$.}\label{example}
\end{figure}

There is a dual CW structure which will also be useful for us. Place a point $V$ in the interior of the $2$-cell above and connect edges from it to the midpoints of the $e_i$; call these edges $\gamma_i$, oriented to the right. This CW complex also has one 0-cell $V$, $n$ $1$-cells $\gamma_1,...,\gamma_n$, and one 2-cell around $P$. Making each $\gamma_i$ smooth at $V$ according to its orientation, we obtain a train track $\tau$ with one switch $V$ and $n=2g$ edges. 

All boundary maps for both complexes are $0$ so the curves $\{\gamma_i\}_{i=1}^{n}$ and $\{e_i\}_{i=1}^n$ can each be seen as representing bases for the integral (co)homology groups of $S$. The Poincar\'e dual to $[\gamma_i]^*$ is $[e_i]$, since $e_j \,\,{\cap\kern-0.4em|}\,\,\,\gamma_i  = +1$ if $i=j$ and $0$ otherwise. Here $\,\,{\cap\kern-0.4em|}\,\,\,$ denotes signed minimal transverse intersection between representatives of the homology classes, and we choose $+1$ here, instead of $-1$, to set a convention. 

The image of $\gamma_i$ is exactly the orbit of the $i^{th}$ rectangle $R_i$. The incidence matrix $A$ of the OBP ($\S$\ref{obps}), defined by $A_{i j}=|B_i\cap O_j|$ represents $f_*$ on $H_1(S;\mathbb{Z})$ in the basis $<[\gamma_1],...,[\gamma_n]>$, while $A^\top$ represents $f^*$ on $H^1(S;\mathbb{Z})$ in the dual basis represented by $<[e_1],...,[e_n]>$.  

\subsection{The Galois Cover}\label{GaloisCover}
We need to construct the Galois cover $p:\widetilde{S}\longrightarrow S$ corresponding to the normal subgroup $K\triangleleft\pi_1(S,V)$ generated by loops that evaluate to $0$ on $f$-invariant cohomology classes. That is, the Galois cover corresponding to the composition \begin{equation*}
\pi_1(S,V) \,\longrightarrow\, H_1(S;\mathbb{Z})\,\longrightarrow\, H= \text{Hom}((H^1(S;\mathbb{Z}))^f,\mathbb{Z}).
\end{equation*}

$\widetilde{S}$ is defined as the space of paths in $S$ starting at $V$ up to the equivalence relation that two paths with the same endpoint are considered the same point of $\widetilde{S}$ if the loop formed by one path followed by the  other backwards evaluates to $0$ on $f$-invariant cohomology. \begin{equation*}
\widetilde{S} = \frac{\{\text{Paths}\,\, \delta:[0,1]\to S \,\,|\,\, \delta(0)=V\}}{\delta_1\sim\delta_2 \iff \delta_1(1)=\delta_2(1) \text{ and }  [\delta_1\cdot\overline{\delta_2}]\in K}
\end{equation*}

The deck group of this cover is $H\cong\pi_1(S,V)/K$, i.e. the homology classes that are dual to the $f$-invariant cohomology of $S$. In terms of the basis $\{[\gamma_1]^*,...,[\gamma_n]^*\}$ of $H^1(S;\mathbb{Z})$, $f$-invariant cohomology is given by the null-space of $(A^\top-I_{n})$. Choose an integral basis for this null-space, \begin{equation*}
(H^1(S;\mathbb{Z}))^f \,\,\cong\,\, \ker(A^\top-I_{n}) \,\, = \,\, <a^1,...,a^{b-1}> .
\end{equation*} 

Here $b = \dim(H^1(M_f;\mathbb{R}))$. The duals $t_i=(a^i)^*$ form a basis for $H$, \[H = \text{Hom}((H^1(S;\mathbb{Z}))^f,\mathbb{Z})\,\, \cong \,\,<t_1,...,t_{b-1}>\,\,\cong\,\mathbb{Z}^{b-1}.\] 

Since $[\gamma_i]^*$ is represented by its Poincar\'e dual loop $[e_i]$ at $P$, each $a^i$ can be written as \begin{equation*}
a^i = \sum_{k=1}^{n}(a^i)_k[e_k].
\end{equation*} 

The $H-$covering space $p:\widetilde{S}\to S$ is thus a $\mathbb{Z}^{b-1}-$sheeted free abelian cover of $S$. The points in $p^{-1}(V)$ can be labelled \begin{equation*}
p^{-1}(V) =  \{ t_1^{c_1}... t_{b-1}^{c_{b-1}}\widetilde{V} \,|\, (c_1,...,c_{b-1})\in\mathbb{Z}^{b-1}\} = \{\mathbf{t}^\mathbf{c}\widetilde{V} \,|\, \textbf{c}\in\mathbb{Z}^{b-1}\}. 
\end{equation*} 

Here, $1\widetilde{V}=\widetilde{V}$ corresponds to the constant path at $V$. Note, as in the introduction, we are letting $\mathbf{t}^\mathbf{c}$ denote $t_1^{c_1}... t_{b-1}^{c_{b-1}}$ for any $\mathbf{c}\in\mathbb{Z}^{b-1}$. Also, denote by $\textbf{a}_i$ the vector formed by the $i^{th}$ coordinates of the basis vectors $a^1,...,a^{b-1}$, \[\textbf{a}_i = ((a^1)_{i},(a^2)_i,...,(a^{b-1})_i)\in\mathbb{Z}^{b-1}.\] 

By definition of the $H$-covering space, for any of the loops $\gamma_i$, its lift to $\widetilde{S}$ starting at $1\widetilde{V}$ ends at $\mathbf{t}^\mathbf{c}\widetilde{V}$ iff $ t_j^*([\gamma_i]) = c_j, \,\forall \,1\leq j\leq b-1$. Hence, for each $ 1\leq j\leq b-1$ and each $ 1\leq i\leq n$,\begin{equation*}
c_j = t_j^*([\gamma_i]) = a^j([\gamma_i]) = \left(\sum_{k=1}^{n}(a^j)_k[e_k]\right)\,\,{\cap\kern-0.4em|}\,\,\,[\gamma_i] = (a^j)_i\,.
\end{equation*}	

In other words, $\gamma_i$ lifts to path in $\widetilde{S}$ from $1\widetilde{V}$ to $\mathbf{t}^{\mathbf{a}_i}\widetilde{V}$. Therefore, we can form $\widetilde{S}$ as follows: Cut the surface $S$ along the edges $e_i$ to obtain a polygon $\overline{S}$ with edges $e_i$ on the boundary, as in Fig. \ref{example}. Take $\mathbb{Z}^{b-1}$ disjoint copies of $\overline{S}$, call them $S_{\mathbf{t}^{\mathbf{c}}}$, enumerated by $\textbf{c}\in\mathbb{Z}^{b-1}$. Each $S_{\mathbf{t}^{\mathbf{c}}}$ thus has two copies of $e_i\subset{S}$, one on the left (oriented clockwise) and one on the right (oriented counter-clockwise). Label the edge in $S_{\mathbf{t}^{\mathbf{c}}}$ corresponding to $e_i$ on the left as ${\mathbf{t}^{\mathbf{c}}}e_i^-$, and on the right as ${\mathbf{t}^{\mathbf{c}}}e_i^+$. For every $1\leq i\leq n$, and every $c\in\mathbb{Z}^{b-1}$ identify ${\mathbf{t}^\mathbf{c}}e_i^+$ with $\mathbf{t}^{\mathbf{a}_i}{\mathbf{t}^{\mathbf{c}}}e_i^-$. 

\begin{equation*}
\widetilde{S} = \frac{\displaystyle\coprod_{\mathbf{c}\in\mathbb{Z}^{b-1}}S_{\mathbf{t}^{\mathbf{c}}}}{\{\,{\mathbf{t}^\mathbf{c}}e_i^+ \sim {\mathbf{t}^{\mathbf{a}_i+\mathbf{c}}}e_i^-\,\,,\,\, \forall \textbf{c}\in\mathbb{Z}^{b-1}, \forall 1\leq i\leq n\,\}}.
\end{equation*} 

\subsection{The Train Track $\widetilde{\tau}$}
The train track $\tau$ on $S$ formed by $V$ and the curves $\gamma_1,...,\gamma_n$ lifts to a train track $\widetilde{\tau}$ on $\widetilde{S}$. We fix the convention that the lift of $\gamma_i$ \underline{\emph{starting}} at $\mathbf{t}^{\mathbf{c}}\widetilde{V}\in S_{\mathbf{t}^{\mathbf{c}}}$ is labelled ${\mathbf{t}^{\mathbf{c}}}\gamma_i$. The terminal point of ${\mathbf{t}^{\mathbf{c}}}\gamma_i$ is then $\mathbf{t}^{\mathbf{a}_i}\mathbf{t}^{\mathbf{c}}\widetilde{V}=\mathbf{t}^{\mathbf{c+a}_i}\widetilde{V}$.

\subsubsection{Example of Fig. \ref{example}, part 2/3:}\label{example2/3}{\footnotesize Let us illustrate this using the admissible OBP $(\sigma,\mathbf{k})$ of Fig. \ref{example}, where the permutation is $\sigma=(2,4,1,6,3,5)$ and $\mathbf{k}=(k_1,...,k_6 )=(11,12,10,12,10,10)$. The blocks are simply obtained by adding the $k_i:$
\[B_1=\{1,...,11\},\,\, B_2=\{12,...,23\},\,\, B_3=\{24,...,33\},\,\, B_4=\{34,...,45\},\,\, B_5=\{46,...,55\},\,\, B_6=\{56,...,66\}.\]

Then, using (\ref{OBPEquation}), the permutation $\xi$ of $\mathbb{N}_{66}$ is computed - it simply permutes the blocks $B_i$ according to $\sigma$. Next, the $\xi-$orbits of each $i\in\mathbb{N}_6$ are computed till their first return to $\mathbb{N}_6$. These are:\\[-3em]

\begin{center}
\begin{align*}
&O_1=(1,11,21,41,61,49,25),  &O_2=(2,12,32,9,19,39,59,47,23,43,63,51,27),\\
&O_3=(3,13,33,10,20,40,60,48,24),&O_4=(4,14,34,54,30,7,17,37,57,45,65,53,29),\\
&O_5=(5,15,35,55,31,8,18,38,58,46,22,42,62,50,26), &O_6=(6,16,36,56,44,64,52,28).
\end{align*}
\end{center}

The induced map $f_*$, computed below as $A_{ij} = |B_i\cap O_j|$ in the basis $\displaystyle\{\gamma_i\}_{i=1}^n$, has characteristic polynomial $(x-1)^4(x^2-11x+1)$. Invariant cohomology, $\ker(f^*-I)=\ker(A^\top-I)$ is generated by the columns of $N$ below: 
\[f_*=A=\left(
\begin{array}{cccccc}
 2 & 2 & 2 & 2 & 2 & 1 \\
 1 & 3 & 2 & 2 & 3 & 1 \\
 1 & 2 & 2 & 2 & 2 & 1 \\
 1 & 2 & 1 & 3 & 3 & 2 \\
 1 & 2 & 1 & 2 & 3 & 1 \\
 1 & 2 & 1 & 2 & 2 & 2 \\
\end{array}
\right),\quad N = \left(
\begin{array}{cc}
 0 & 0 \\
 0 & 0 \\
 -1\,\,\,\, & -1\,\,\,\, \\
 0 & 0 \\
 0 & 1 \\
 1 & 0 \\
\end{array}
\right)\hspace{1cm}\]

Thus, we can take $H^1(S;\mathbb{Z})^f \,\,\cong\,\, \,<a^1, a^2>\,\,\, := \,\,\,<[\gamma_6]^*-[\gamma_3]^*,\,[\gamma_5]^*-[\gamma_3]^*> \,\,\,\cong\,\, \,<[e_6]-[e_3],\,[e_5]-[e_3]>$. The generators of $H$ are $t_1, t_2$, dual to $a^1, a^2$ respectively. Thus, the curves $1\gamma_1, 1\gamma_2$, and $1\gamma_4$ are loops at $\widetilde{V}$, as the corresponding rows of $N$, namely $\textbf{a}_1=\textbf{a}_2=\textbf{a}_4=\mathbf{0}$. $1\gamma_6$ is the path from $\widetilde{V}\in S_1\,$ to $\,t_1\widetilde{V}\in S_{t_1}$; $1\gamma_5$ ends at $t_2\widetilde{V}\in S_{t_2}$; and $1\gamma_3$ ends at $t_1^{-1}t_2^{-1}\widetilde{V}\in S_{t_1^{-1}t_2^{-1}}$. }

\subsection{Computing the Teichm\"uller Polynomial}
Given $n=2g\geq4$ and an admissible pair $(\sigma,\mathbf{k})$ consisting of a permutation $\sigma$ of $\mathbb{N}_n$ and positive integers $\mathbf{k}=(k_1,...,k_n)$, one defines via (\ref{OBPEquation}) the ordered block permutation $\xi$. Next one computes $O_i$, the $\xi-$orbit of $i$, $\forall\,\,1\leq i\leq n$.
\begin{equation*}
O_i := (i,\xi(i),\xi(\xi(i)),...,\xi^{\circ(m_i-1)}(i)),
\end{equation*} where $m_i>1$ is the smallest integer such that $\xi^{\circ m_i}(i)\leq n$.

Using the block function $\beta$ to see which block each entry of the orbit $O_i$ belongs to, one obtains the images of $\gamma_i$ under the oriented-fixed map $f=f_{\sigma,\mathbf{k}}$ associated to the OBP. The image $f(\gamma_i)$ is homotopic to a concatenation of the $\{\gamma_j\}$. We denote it as an ordered list,\begin{equation*}
f(\gamma_i) \simeq (\gamma_{\beta(i)},\gamma_{\beta(\xi(i))},...,\gamma_{\beta(\xi^{\circ(m_i-1)}(i))}).
\end{equation*}

For better readability, let us relabel the image of $\gamma_i$ under $f$ as \begin{equation*}
f(\gamma_i) \simeq (\gamma_{i_1},\gamma_{i_2},...,\gamma_{i_{m_i}}).
\end{equation*}

We can choose the lift $\widetilde{f}:\widetilde{S}\to\widetilde{S}$ of $f$ by stipulating that $\widetilde{f}$ fix the point $1\widetilde{V}$. So, the image $\widetilde{f}(1\gamma_i)$ starts with the path $1\gamma_{i_1}$, connecting $1\widetilde{V}$ to ${\mathbf{t}^{\mathbf{a}_{i_1}}}\widetilde{V}$. The next curve in $\widetilde{f}(1\gamma_i)$ must therefore start at $\mathbf{t}^{\mathbf{a}_{i_1}}\widetilde{V}$, so it's the curve $\mathbf{t}^{\mathbf{a}_{i_1}}\gamma_{i_2}$, which terminates at ${\mathbf{t}^{\mathbf{a}_{i_2}}}{\mathbf{t}^{\mathbf{a}_{i_1}}}\widetilde{V}$. That is, the first element of the orbit affects where the second element starts, the first and second affect the label of the third element, and so on. With this, we can form the orbit $\widetilde{f}(1\gamma_i)$ as the concatenation of the curves 

\begin{equation*}
\widetilde{f}(1\gamma_i) \simeq \left( 1\gamma_{i_1},\,\,  \textbf{t}^{\textbf{a}_{i_1}} \gamma_{i_2},\,\,  \textbf{t}^{\textbf{a}_{i_2}}  \textbf{t}^{\textbf{a}_{i_1}} \gamma_{i_3},\,...\,,\prod_{l=1}^{m_i-1}\textbf{t}^{\textbf{a}_{i_l}}\gamma_{i_{m_i}}\right)
\end{equation*}

The free $\mathbb{Z}[H]$-module formed by the edges of $\widetilde{\tau}$ is generated by $\{1\gamma_1,...,1\gamma_n\}$. The action of $\widetilde{f}$ on this module can be encoded in the matrix $A(\mathbf{t})$ of Laurent polynomials in $\mathbf{t}=(t_1,...,t_{b-1})$.  The entry $\prod_{l=1}^{j-1}\textbf{t}^{\textbf{a}_{i_l}}\gamma_{i_j}$ of $\widetilde{f}(1\gamma_i)$ contributes $\prod_{l=1}^{j-1}\textbf{t}^{\textbf{a}_{i_l}}$ to the $({i_j,i})^{th}$ entry of $A(\mathbf{t})$. 

\subsection{Proof of Theorem \ref{mainThm}}
Our main result, Theorem \ref{mainThm}, states that the characteristic polynomial of $A(\mathbf{t})$ computed as above for the OBP $(\sigma,\mathbf{k})$ contains the Teichm\"uller polynomial associated to the mapping torus of the pseudo-Anosov map $f_{\sigma,\mathbf{k}}$. We show this now.

\begin{proof}[Proof of Theorem \ref{mainThm}]
We just need to show that $\frac{1}{u-1}\det(uI_n-A(\mathbf{t}))$ is the polynomial we would get as the output from Theorem \ref{McThm} using a minimal trivalent train track $\tau'$ carrying the foliation of $f$.

\begin{figure}[h!]
\includegraphics[width=.9\textwidth]{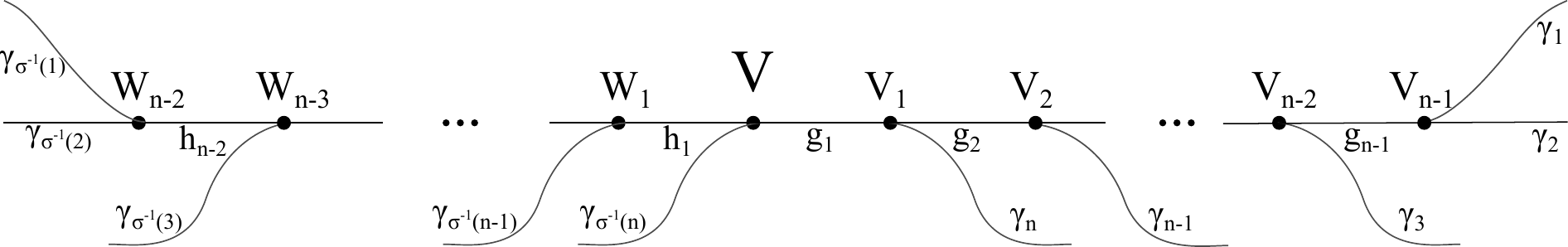}
\caption{A modification of the train track $\tau$ in a neighborhood of the only vertex $V$ of $\tau$, such that the modified train track $\tau'$ is trivalent.}
\label{TTrack}
\end{figure}

We construct $\tau'$ out of $\tau$ by modifying $\tau$ locally in a neighborhood of its only vertex $V$, so as to make it trivalent in a particular way. Insert a short edge $g_1$ to the right of $V$ before $\gamma_n$ branches off to the right at a vertex defined to be $V_1$. After another short edge $g_2$, $\gamma_{n-1}$ branches off at $V_2$. Continuing this way, the last edge to add would be $g_{n-1}$ until $\gamma_2$ branches off at $V_{n-1}$.

Similarly we add edges to the left of $V$, but the first one to branch off (to the left) is $\gamma_{\sigma^{-1}(n)}$ at $V$ itself. Then, after a short edge $h_1$, $\gamma_{\sigma^{-1}(n-1)}$ branches off at vertex $W_1$, and so on, till the last to branch off at $W_{n-2}$ after the edge $h_{n-2}$ is $\gamma_{\sigma^{-1}(2)}$. We thus have $n+(n-1)+(n-2) = 3(n-1)=3(2g-1)$ edges, and $1+(n-1)+(n-2)=2(2g-1)$ vertices. 

This is the minimal number of edges as follows: any train-track with at least one loop that is invariant under a pseudo-Anosov map on a surface of genus $g$ must have Euler characteristic at least $2g-1$. So its $\#|edges|-\#|vertices|\geq 2g-1$. If it is trivalent, $2\#|edges|=3\#|vertices|$, so we get $\#|edges|\geq 3(3g-1)$.

In a small neighborhood of $V$, $f$ stretches horizontally and shrinks vertically by the stretch-factor, and raises the scaled neighborhood up towards the fixed singularity $P$ at the top of the zippered polygon (as in Figure \ref{example}). Moreover, OBPs satisfy $k_1\geq n$ \cite[Lemma 4.4]{HRS19}, which implies that each $f(\gamma_i)$ traverses $\gamma_1$ as the first element of its edge-path. Similarly, $k_{\sigma^{-1}(1)}\geq n$, and the terminal segment of each $f(\gamma_i)$ is $\gamma_{\sigma^{-1}(1)}$. Hence, by choosing the lengths of the new edges $g_j,h_k$ to be small, the image $f(\tau')$ can be homotoped onto $\tau'$ in a way that all the vertices and the new edges map to themselves $-$ this is why we modified $\tau$ in the order shown in Figure \ref{TTrack}. 

Take as basis of the free $\mathbb{Z}[H]$-modules of the edges and vertices of the lift $\widetilde{\tau'}$ to be,
\[E=\{1\gamma_1,...,1\gamma_n,1g_1,...,1g_{n-1},1h_1,...,1h_{n-2}\},\]
\[F=\{1\widetilde{V},1\widetilde{V_1},...,1\widetilde{V_{n-1}},1\widetilde{W_1},...,1\widetilde{W_{n-2}}\}.\]

Since $\widetilde{V}$ is fixed by $\widetilde{f}$, each $1g_j$, each $1h_k$, and each vertex in $F$ is fixed by the train track map induced by $\widetilde{f}$ on $\widetilde{\tau'}$. Therefore, the matrix $P_E(\mathbf{t})$ representing the action of $\widetilde{f}$ on the free $\mathbb{Z}[H]-$module generated edges of $\widetilde{\tau'}$ has the form
\begin{equation*}
P_E(\mathbf{t}) = \begin{pmatrix}
A(\mathbf{t}) & 0\\
* & I_{2n-3}
\end{pmatrix}.
\end{equation*}
Moreover, $P_F(t)=I_{2n-2}$. Hence,\begin{equation*}
\frac{\det(uI-P_E(\mathbf{t}))}{\det(uI-P_F(\mathbf{t}))} = \frac{(u-1)^{2n-3}\det(uI-A(\mathbf{t}))}{(u-1)^{2n-2}} = \frac{\det(uI-A(\mathbf{t}))}{(u-1)}.
\end{equation*}
\end{proof}

\begin{remark}\label{rmk1}
With slight modifications, this procedure works more generally when the number of singularities $\nu$ is more than one. One gets a \,$(2g+\nu-1)-$dimensional matrix $A$, in terms of a spanning set, not a basis for $H_1(S)$. One just has to be careful not to pick something in the null-space of $A^\top-I_n$ that corresponds to the extra dimensions rather than invariant cohomology. We decided to keep this paper short since for our application one singularity suffices.
\end{remark}

\begin{remark}
Using a train track with a single switch instead of a trivalent train track also seems to work more generally. The proof above needs just slight modification as long as the foliations are orientable and one contracting singular segment is fixed.
\end{remark}

\subsubsection{The example of Fig. \ref{example}, part 3/3:}\label{example3/3}
{\footnotesize We will finish this section by computing the Teichm\"uller polynomial of our running example. Using the block function $\beta:\mathbb{N}_{66}\to\mathbb{N}_6$ to determine which block each element of the orbit $O_i$ (computed in $\S$\ref{example2/3}) belongs to, we obtain the images of the $\gamma_i$ as edge-paths:
\begin{alignat*}{2}
&f(\gamma_1)\simeq(\gamma_1,\gamma_1,\gamma_2,\gamma_4,\gamma_6,\gamma_5,\gamma_3),  &f(\gamma_2)\simeq(\gamma_1,\gamma_2,\gamma_3,\gamma_1,\gamma_2,\gamma_4,\gamma_6,\gamma_5,\gamma_2,\gamma_4,\gamma_6,\gamma_5,\gamma_3),\\
&f(\gamma_3)\simeq(\gamma_1,\gamma_2,\gamma_3,\gamma_1,\gamma_2,\gamma_4,\gamma_6,\gamma_5,\gamma_3),&f(\gamma_4)\simeq(\gamma_1,\gamma_2,\gamma_4,\gamma_5,\gamma_3,\gamma_1,\gamma_2,\gamma_4,\gamma_6,\gamma_4,\gamma_6,\gamma_5,\gamma_3),\\
&f(\gamma_5)\simeq(\gamma_1,\gamma_2,\gamma_4,\gamma_5,\gamma_3,\gamma_1,\gamma_2,\gamma_4,\gamma_6,\gamma_5,\gamma_2,\gamma_4,\gamma_6,\gamma_5,\gamma_3), &f(\gamma_6)\simeq(\gamma_1,\gamma_2,\gamma_4,\gamma_6,\gamma_4,\gamma_6,\gamma_5,\gamma_3).
\end{alignat*}

For the lifted train track, note that $\gamma_1, \gamma_2,\gamma_4$ lift to loops at each vertex $\mathbf{t^c}\widetilde{V}$. Traversing a lift of $\gamma_6$ takes one from $\mathbf{t^c}\widetilde{V}$ to $\,t_1\mathbf{t^c}\widetilde{V}$, so $\gamma_6$ changes the label of the elements \emph{after} it by multiplying them by $t_1$. Similarly,  each $\gamma_5$ multiplies the next elements of the orbit by $t_2$, and $\gamma_3$ by $t_1^{-1}t_2^{-1}$. Using this, we can compute the induced action of $\widetilde{f}$ on $\widetilde{\tau}$.

For instance, $\widetilde{f}(\textcolor{Red}{1}\gamma_1)\simeq (\textcolor{Red}{1}\gamma_1,\,\,\textcolor{Red}{1}\gamma_1,\,\, \textcolor{Red}{1}\gamma_2,\,\, \textcolor{Red}{1}\gamma_4,\,\, \textcolor{Red}{1}\gamma_6,\,\, \textcolor{Red}{t_1}\gamma_5,\,\, \textcolor{Red}{t_1t_2}\gamma_3)$. Note that the last edge $\textcolor{Red}{t_1t_2}\gamma_3$ terminates at $1\widetilde{V}$, as it should, as $1\gamma_1$ is a loop in $\widetilde{S}$ and $\widetilde{f}$ is a homeomorphism. Similarly computing $\widetilde{f}(\textcolor{Red}{1}\gamma_2),...,\widetilde{f}(\textcolor{Red}{1}\gamma_6)$, we get $A(\mathbf{t})$ and the Teichm\"uller polynomial in this case as
\[
A(\mathbf{t})=A(t_1,t_2)=\left(
\begin{array}{cccccc}
 2 & \frac{1}{t_1 t_2}+1 & \frac{1}{t_1 t_2}+1 & \frac{1}{t_1}+1 & \frac{1}{t_1}+1 & 1 \\
 1 & \frac{1}{t_1 t_2}+2 & \frac{1}{t_1 t_2}+1 & \frac{1}{t_1}+1 & t_2+\frac{1}{t_1}+1 & 1 \\
 t_1 t_2 & t_1 t_2+1 & 2 & t_1 t_2+t_2 & t_1 t_2^2+t_2 & t_1^2 t_2 \\
 1 & \frac{1}{t_1 t_2}+1 & \frac{1}{t_1 t_2} & \frac{1}{t_1}+2 & t_2+\frac{1}{t_1}+1 & t_1+1 \\
 t_1 & t_1+\frac{1}{t_2} & \frac{1}{t_2} & t_1+1 & t_1 t_2+2 & t_1^2 \\
 1 & \frac{1}{t_1 t_2}+1 & \frac{1}{t_1 t_2} & \frac{1}{t_1}+1 & t_2+\frac{1}{t_1} & t_1+1 \\
\end{array}
\right),\]\[
\Theta_F(t_1,t_2,u)=\frac{\det(uI_6-A(\mathbf{t}))}{u-1}=(u-1)^3\left(u^2-\left(t_1t_2+t_1+7+\frac{1}{t_1}+\frac{1}{t_1t_2}\right)u+1\right).
\]
}

\section{Sequences of Teichm\"uller Polynomials}\label{deg4}
As an application of the procedure above, we show that for every $g\geq2$ and $p\geq0$, the polynomials $\Theta_{g,p}$ given in (\ref{polys}) are the Teichm\"uller polynomials associated to oriented-fixed pseudo-Anosov maps $f_{g,p}$, whose mapping tori have first Betti number $b_1=g$. By evaluating $\Theta_{g,p}$ at specific values within its fibered cone, we will deduce Proposition \ref{Prop1}. \begin{equation*}
\Theta_{g,p}(t_1,...,t_{g-1},u) = (u-1)^{2g-3}\left(u^2-(\sum_{i=1}^{g-1}t_i+2g+p+1+\sum_{i=1}^{g-1}\frac{1}{t_i})u+1\right)
\end{equation*}

\subsection{The Polynomials $\Theta_{g,0}$}\label{secg0} First, let us consider the case $p=0$. We'll show $\Theta_{g,0}$ is the Teichm\"uller polynomial associated to the OBP $(\sigma, \mathbf{k})$ of size $n=2g$ given by,\begin{equation*}
\begin{aligned}	
\sigma&=(\sigma(1),...,\sigma(n)) = \,(n,..., 2, 1),\\
\mathbf{k}&=(k_1,...,k_{n}) = \, (2n-1,...,2n-1,n).
\end{aligned}\end{equation*} 

Let us first check that the pair $(\sigma, \mathbf{k})$ satisfies the admissibility conditions (Definition\,\ref{admissibledef}). For a clearer exposition, let $d:= 2n-1 = 4g-1$. 

\begin{lemma}\label{skadmis} The pair $(\sigma,\mathbf{k})$ above defines an admissible OBP.\end{lemma}
\begin{proof} Here $K = \sum_{i=1}^nk_i = (n-1)(2n-1)+n = (n-1)d+n$. The blocks are { \begin{align*} B_1=\{1,..., d\}, \quad B_2&=\{d+1,...,2d\},...\\
  ... ,\,B_{n-1}&=\{(n-2)d+1,...,(n-1)d\},\quad B_n=\{(n-1)d+1,...,(n-1)d+n\}.
\end{align*}}

Using (\ref{OBPEquation}), we find the associated OBP $\xi:\mathbb{N}_K\to\mathbb{N}_K$ to be:
{\[ \xi(j) = \begin{cases} 
      j+n+(n-2)d \quad & 1\leq j\leq d \\
      j+n+(n-4)d & d+1\leq j\leq 2d \\
      \hspace{1cm}\vdots & \\
      j+n+(n-2(n-1))d \quad & (n-2)d+1\leq j\leq (n-1)d\\
      j+(1-n)d & (n-1)d+1\leq j\leq (n-1)d+n = K
   \end{cases}
\]}

Next, we compute the $\xi-$orbits $O_i$ of the first $n$ elements until their first return to $\mathbb{N}_n$.\\[-.5em]

{\small $O_1 = (\overline{\overline{1}},\,\,\,  (n-2)d+n+1,\,\,\, d + 2, \,\,\, (n-3)d+n+2, \,\,\, 2d + 3, \,\,\, (n-4)d+n+3, \,\,\,\ldots$\\ 
					
$\hfill \ldots,\,\,\,(n-2)d+n-1,\,\,\, \underline{\underline{n+(n-1)}} = d,\,\,\, (n-1)d+n = K \,\,)$.\\
					
$O_2 = (2,\,\,\,  (n-2)d+n+2,\,\,\, d + 3,\,\,\, (n-3)d+n+3, \,\,\,2d + 4,\,\,\, {\ldots}$\\

$\hfill\ldots,\,\,\,   \underline{\underline{d + n + (n-1)}}=2d,\,\,\, (n-2)d+n,\,\,\,
					\overline{\overline{n+n}}=d+1,\,\,\,(n-3)d+n+1,\,\,\, 2d+2,\,\,\,\ldots\hfill $\\ 
					
$\hfill\ldots,\,\,\, (n-2)d+(n-2),\,\,\,n+(n-2),\,\,\,(n-1)d+n-1=K-1 \,\,)$.\\
					
$O_3 = (3, \,\,\, (n-2)d+n+3,\,\,\, d + 4,\,\,\, (n-3)d+n+4,\,\,\, 2d + 5,\,\,\,{\ldots}$\\

$\hfill\ldots,\,\,\,\underline{\underline{2d + n + (n-1)}}=3d,\,\,\, (n-3)d+n,\,\,\,\overline{\overline{d+n+n}}=2d+1,\,\,\,(n-4)d+n+1,\,\,\, 3d+2,\,\,\,\ldots\hfill$\\ 
					
$\hfill  \ldots,\,\,\, (n-2)d+(n-3),\,\,\,n+(n-3),\,\,\,(n-1)d+n-2=K-2 \,\,)$.	
					
${ \vdots}$\\
					
$O_{n-2} = (n-2,\,\,\, (n-2)d+n+(n-2), \,\,\,d+(n-1),\,\,\, \underline{\underline{(n-2)d}},\,\,2d+n,$\\

$\hfill\overline{\overline{(n-3)d+1}},\,\,\,d+n+1,\,\,\,(n-2)d+2,\,\,\,n+2,\,\,\,(n-1)d+3 = K-(n-3))$. \\

$O_{n-1} = (n-1,\,\, \underline{\underline{(n-1)d}},\,\,d+n,\,\,\overline{\overline{(n-2)d+1}},\,\,n+1,\,\,(n-1)d+2 = K-(n-2))$. \\
		
$O_n=(n,\,\,\,\overline{\overline{(n-1)d+1}} = K-(n-1)).$
}\\[-.5em]

Each orbit $O_i$ ends at $K+1-i$. As $1\leq i\leq n$, we have $(n-1)d+1\leq K+1-i\le K$. Thus the first return map $\xi'(i)=\xi(K+1-i) = K+1-i-(n-1)d=n+1-i=\sigma(i)$, as required by Def.\ref{admissibledef}(\textsc{i}). 

To verify Def.\ref{admissibledef}(\textsc{ii}), note that $O_1$ has $2n-1$ elements, whereas for $2\leq i\leq n, |O_i|=4(n-i)+2$. Hence the sum of all orbits is $(2n-1)+\sum_{i=2}^n(4(n-i)+2) = (n-1)(2n-1)+n=K$. Since $\xi$ is a permutation, $\coprod_{i=1}^nO_i=\mathbb{N}_K$. 

To verify Def.\ref{admissibledef}(\textsc{iii}), we've underlined the last and over-lined the first element of the block $B_i$ in each $O_i$ - except that the last element of $B_n$, which is $K$, is in $O_{\sigma^{-1}(n)}=O_1$ as required. It may help the reader to note that orbits $O_i$ and $O_{i+1}$ remain adjacent until the underlined element of $O_i$ (and the over-lined element of $O_{i+1}$), after which they diverge. 
		
Def.\ref{admissibledef}(\textsc{iv}) requires the incidence matrix defined by $A_{i\,j}=|O_j\cap B_i|$ to be irreducible. Note that the first element of each orbit $O_i$ is $i\in B_1$, so the first row of $A$ is positive. Also, the odd elements of the first orbit $O_1$ are $\{1, d+2, 2d+3,...,(n-1)d+n\}$, one in each block $B_i$, so the first column of $A$ is also positive. Hence $A^2>0$, so $A$ is irreducible. 

Thus, the OBP $(\sigma,\mathbf{k})$ is admissible.\end{proof} 

Using the methods in \cite[$\S$5]{HRS19}, the permutation $\sigma$ yields a surface with one $(2n-2)-$pronged singularity when $n$ is even. Since the pair $(\sigma,\mathbf{k})$ is admissible, by \cite[Theorem 6.1]{HRS19} we obtain oriented-fixed pseudo-Anosov homeomorphisms $f_{g,0}:S_g\to S_g$ for each $g\geq2$ whose induced action on the homology in terms of the basis $\{[\gamma_1],...,[\gamma_n]\}$ (defined in $\S$\ref{algo} above) is represented by the matrix $A$.  Let us use the block-function $\beta$, which in this case is simply $\beta(j)=\cl{j/d}$, to turn the orbits $O_i$ into edge-paths $f(\gamma_i)$:

{ 
\begin{align}
\begin{split}\label{p0orbits}
&f(\gamma_1) \simeq (\overline{\overline{\gamma_1}},\,  \gamma_{n-1},\,\,\, \gamma_2, \,
						 \gamma_{n-2},\,\,\, \gamma_3, \,
						 \gamma_{n-3},\,\,\, 
						{\ldots}\,\,\,,\,\,\,
						 \gamma_{n-2},\, \gamma_2,\,\,\, \gamma_{n-1},\,
					\underline{\underline{\gamma_1}}, \,\,\,\gamma_n \,).\\
&f(\gamma_2) \simeq (\gamma_1,\,  \gamma_{n-1},\,\,\, \gamma_2,\,
						 \gamma_{n-2},\,\,\,\gamma_3, \,\,\, 
						{\ldots}\,\,\,,\,\,\,
						 \underline{\underline{\gamma_2}},\, \gamma_{n-1},\,\,\,\overline{\overline{\gamma_2}},\, \gamma_{n-2}, \,\,\,\gamma_3,\,\,\,
						 {\ldots}\,\,\,,\,\,\,
						  \gamma_{n-1},\,\,\,
					\gamma_1,\,\,\, \gamma_n \,).\\
&f(\gamma_3) \simeq (\gamma_1,\,  \gamma_{n-1},\,\,\, \gamma_2, \,
						 \gamma_{n-2},\, \,\,\gamma_3, \,\,\,
						{\ldots}\,\,\,,\,\,\,
						 \underline{\underline{\gamma_3}},\, \gamma_{n-2},\,\,\,\overline{\overline{\gamma_3}},\, \gamma_{n-3},\,\,\, \gamma_4,\,\,\,
						 {\ldots}\,\,\,,\,\,\,
						\gamma_{n-1},\,\,\,
					\gamma_1,\,\,\, \gamma_n \,).\\
&\vdots\\
&f(\gamma_{n-2}) \simeq (\gamma_1,\, \gamma_{n-1}, \,\,\,\gamma_2,\, \underline{\underline{\gamma_{n-2}}},\,\,\gamma_3,\, \overline{\overline{\gamma_{n-2}}},\,\,\, \gamma_2,\, \gamma_{n-1},\,\,\,\gamma_1,\,\,\,\gamma_n). \\
&f(\gamma_{n-1}) \simeq (\gamma_1,\,  \underline{\underline{\gamma_{n-1}}},\,\,\, \gamma_2, \,
	 \overline{\overline{\gamma_{n-1}}},\,\,\, \gamma_1, \,\,\, \gamma_n \,). \\
&f(\gamma_{n}) \simeq (\gamma_1,\,\,\, \overline{\overline{\gamma_n}} \,).	
\end{split}
\end{align}	
}\\[-1em]
		
Notice that in $f(\gamma_1)$, the odd entries are $\gamma_1,\gamma_2,...,\gamma_n$, while the even entries interspersed between them are $\gamma_{n-1},...,\gamma_1$. For $f(\gamma_2)$, the odd entries are $\gamma_1,...,\gamma_{n-1},...,\gamma_1$, and the even entries are $\gamma_{n-1},...,\gamma_2,\gamma_2,...,\gamma_n$. Generally, for $2\leq i\leq n$, the odd entries of $f(\gamma_i)$ are $\gamma_1,..., \gamma_{n-i+1},...,\gamma_1$, while the even entries are $\gamma_{n-1},..., \gamma_i,\gamma_i,...,\gamma_n$. From this, we find that the incidence matrix is given by $A_n$ below. 
		
\begin{equation*}
A_n = \left(
\begin{array}{ccccccccccc}
 2 & 2 & 2 & \cdots & 2 & 2 & \cdots & 2 & 2 & 1 \\
 2 & 4 & 2 & \cdots & 2 & 2 & \cdots & 2 & 1 & 0 \\
 2 & 4 & 4 & \cdots & 2 & 2 & \cdots & 1 & 0 & 0 \\
 & \vdots && \ddots &   &   & \udots && \vdots & \\
 2 & 4 & 4 & \cdots & 4 & 1 & \cdots & 0 & 0 & 0 \\
 2 & 4 & 4 & \cdots & 3 & 2 & \cdots & 0 & 0 & 0 \\
 & \vdots && \udots &   &   & \ddots && \vdots & \\
 2 & 4 & 3 & \cdots & 2 & 2 & \cdots & 2 & 0 & 0 \\
 2 & 3 & 2 & \cdots & 2 & 2 & \cdots & 2 & 2 & 0 \\
 1 & 1 & 1 & \cdots & 1 & 1 & \cdots & 1 & 1 & 1 \\
\end{array}
\right)\end{equation*}

Note that $A_n-I_n$ is almost symmetric across the middle horizontal line: the $2^{nd}$ row is the same as the $(n-1)^{st}$ row and so on until the middle two rows are the same. Clearing rows $2$ through $g=n/2$, its easy to see that $A_n-I_n$ has rank $g+1$. The null-space of $(A_n^\top-I_n)$ is thus $g-1$ dimensional, generated by the columns of $N$ below, \begin{equation}\label{invariantCoh}
N=\left(
\begin{array}{ccccccccccc}
 0 & 0 &  \cdots &  0  \\
-1\,\,\,\, & 0 &  \cdots &  0  \\
 0 & -1\,\,\,\, & \cdots &  0  \\
\vdots & \vdots &  & \vdots     \\
 0 & 0  & \cdots &  -1\,\,\,\,  \\
 0 & 0  & \cdots &  1  \\
\vdots & \vdots &   &  \vdots \\
 0 & 1  & \cdots  & 0 \\
 1 & 0  & \cdots  & 0 \\
 0 & 0  & \cdots  & 0  \\
\end{array}
\right)_{n\times(g-1)}
\end{equation}

Let $a^1,...,a^{g-1}$ represent the columns of $N$, and $t_1,...,t_{g-1}\in H=\,$Hom$(H^1(S_g;\mathbb{Z})^{f_g};\mathbb{Z})$ their respective duals. So, $H\cong\mathbb{Z}^{g-1}$ and the first Betti number of the mapping torus of $f_g$ is $g$.

Consider the $H$-covering space $p:\widetilde{S}\to S$ constructed in $\S$\ref{GaloisCover}, where we denote the lift of $\gamma_i$ starting at $1\widetilde{V}$ in the sheet $S_1\subset\widetilde{S}$ by $1\gamma_i$. As shown there, $1\gamma_i$ terminates at $\mathbf{t}^{\mathbf{a}_i}\widetilde{V}$ in the sheet $S_{\mathbf{t}^{\mathbf{a}_i}}$, where $\textbf{a}_i$ denotes the $i^{th}$ \emph{row} of $N$. 

Thus $1\gamma_1$ and $1\gamma_n$ are loops at $\widetilde{V}$. But $1\gamma_2$ is the path from $1\widetilde{V}$ to $t_1^{-1}\widetilde{V}$; $1\gamma_3$ terminates at $t_2^{-1}\widetilde{V};$ and so on until $1\gamma_g$ terminates at $t_{g-1}^{-1}\widetilde{V}$. Similarly, $1\gamma_{g+1}$ terminates at $t_{g-1}\widetilde{V}$; $1\gamma_{g+2}$ at $t_{g-2}\widetilde{V}$; and so on until $1\gamma_{n-1}$ which terminates at $t_1\widetilde{V}$. This allows us to write the images under $\widetilde{f}$ as the concatenation of the the following paths:\\[0em]

{ \noindent$\widetilde{f}(\textcolor{Red}{1}\gamma_1) = (\textcolor{Red}{1}\gamma_1,\, \textcolor{Red}{1} \gamma_{n-1},\,\,\, \textcolor{Red}{t_1}\gamma_2, \,
						 \textcolor{Red}{1}\gamma_{n-2},\,\,\,\textcolor{Red}{t_2} \gamma_3, \,
						 \textcolor{Red}{1}\gamma_{n-3},\,\,\, 
						{\ldots}\,\,\,,\,\,\,
						 \textcolor{Red}{t_2^{-1}}\gamma_{n-2},\, \textcolor{Red}{1}\gamma_2,\,\,\, \textcolor{Red}{t_1^{-1}}\gamma_{n-1},\,
					\textcolor{Red}{1}\gamma_1, \,\,\,\textcolor{Red}{1}\gamma_n \,)$.\\[-.5em]
					
	\noindent$\widetilde{f}(\textcolor{Red}{1}\gamma_2) = (\textcolor{Red}{1}\gamma_1,\,  \textcolor{Red}{1}\gamma_{n-1},\,\,\, \textcolor{Red}{t_1}\gamma_2,\,
						 \textcolor{Red}{1}\gamma_{n-2},\,\,\,\textcolor{Red}{t_2}\gamma_3, \,\,\, 
						{\ldots}\,\,\,,\,\,\,
						\textcolor{Red}{1} \gamma_2,\, \textcolor{Red}{t_1^{-1}}\gamma_{n-1},$\\[-.5em]
						
						$\hfill\,\,\,\textcolor{Red}{1}\gamma_2,\, \textcolor{Red}{t_1^{-1}}\gamma_{n-2}, \,\,\,\textcolor{Red}{t_1^{-1}}\textcolor{Red}{t_2}\gamma_3,\,\,\,
						 {\ldots}\,\,\,,\,\,\,
						  \textcolor{Red}{t_1^{-2}}\gamma_{n-1},\,\,\,
					\textcolor{Red}{t_1^{-1}}\gamma_1,\,\,\, \textcolor{Red}{t_1^{-1}}\gamma_n \,)$.\\[-.5em]
					
	\noindent$\widetilde{f}(\textcolor{Red}{1}\gamma_3) = (\textcolor{Red}{1}\gamma_1,\,  \textcolor{Red}{1}\gamma_{n-1},\,\,\, \textcolor{Red}{t_1}\gamma_2, \,
						 \textcolor{Red}{1}\gamma_{n-2},\, \,\,\textcolor{Red}{t_2}\gamma_3, \,\,\,
						{ \ldots}\,\,\,,\,\,\,
						 \textcolor{Red}{1}\gamma_3,\, \textcolor{Red}{t_2^{-1}}\gamma_{n-2},$\\[-.5em]
						
						$\hfill\,\,\,\textcolor{Red}{1}\gamma_3,\, \textcolor{Red}{t_2^{-1}}\gamma_{n-3},\,\,\, \textcolor{Red}{t_2^{-1}t_3}\gamma_4,\,\,\,
						 { \ldots}\,\,,\,
						\textcolor{Red}{t_1^{-1}t_2^{-1}}\gamma_{n-1},\,
					\textcolor{Red}{t_2^{-1}}\gamma_1,\, \textcolor{Red}{t_2^{-1}}\gamma_n \,)$.\\[-.5em]
					
					$\vdots$\\[-.5em]
					
					\noindent$\widetilde{f}(\textcolor{Red}{1}\gamma_{n-2}) = (\textcolor{Red}{1}\gamma_1,\, \textcolor{Red}{1}\gamma_{n-1}, \,\,\,\textcolor{Red}{t_1}\gamma_2,\, \textcolor{Red}{1}\gamma_{n-2},\,\,\textcolor{Red}{t_2}\gamma_3,\, \textcolor{Red}{1}\gamma_{n-2},\,\,\, \textcolor{Red}{t_2}\gamma_2,\, \textcolor{Red}{t_1^{-1}t_2}\gamma_{n-1},\,\,\,\textcolor{Red}{t_2}\gamma_1,\,\,\,\textcolor{Red}{t_2}\gamma_n)$. \\[-.5em]
					
		\noindent$\widetilde{f}(\textcolor{Red}{1}\gamma_{n-1}) = (\textcolor{Red}{1}\gamma_1,\,  \textcolor{Red}{1}\gamma_{n-1},\,\,\, \textcolor{Red}{t_1}\gamma_2, \,
						 \textcolor{Red}{1}\gamma_{n-1},\,\,\, \textcolor{Red}{t_1}\gamma_1, \,\,\, \textcolor{Red}{t_1}\gamma_n \,)$. \\[-.5em]
		
		\noindent$\widetilde{f}(\textcolor{Red}{1}\gamma_{n}) = (\textcolor{Red}{1}\gamma_1,\,\,\, \textcolor{Red}{1}\gamma_n \,).$					}\\[0em]

Thus, the matrix $A_n(\mathbf{t})=A_{2g}(t_1,...,t_{g-1})$ representing the action of $\widetilde{f}$ on the $\mathbb{Z}[t_1^{\pm1},...,t_{g-1}^{\pm1}]$-module formed by the edges of the lifted train track has the following form:

{\footnotesize
\begin{equation*}
\left(
\begin{array}{ccccccccccc}
 2 				& 1+\frac{1}{t_1} 					& 1+\frac{1}{t_2} 								& \cdots & 1+\frac{1}{t_{g-1}} 							& 1+t_{g-1} 			& \cdots & 1+t_2 	& 1+t_1 & 1 \\
 &&&&&&&&&\\
 1+t_1 			& 2+t_1+\frac{1}{t_1}				& t_1+\frac{1}{t_2} 							& \cdots & t_1+\frac{1}{t_{g-1}} 						& t_1+t_{g-1} 			& \cdots & t_1+t_2 	& t_1	& 0 \\
 &&&&&&&&&\\
 1+t_2 			& (1+\frac{1}{t_1})(1+t_2) 			& 2+t_2+\frac{1}{t_2}	 						& \cdots & t_2+\frac{1}{t_{g-1}} 						& t_2+t_{g-1} 			& \cdots & t_2 		& 0 	& 0 \\
 &&&&&&&&&\\
 \vdots 		& \vdots      						&	\vdots 										& \ddots & \vdots  					 					& \vdots 				& \udots &\vdots	&\vdots &\vdots \\	
 &&&&&&&&&\\
 1+t_{g-1}      & (1+\frac{1}{t_1})(1+t_{g-1}) 		& (1+\frac{1}{t_2})(1+t_{g-1}) 					& \cdots & 2+t_{g-1}+\frac{1}{t_{g-1}} 					& t_{g-1} 				& \cdots & 0 		& 0 	& 0 \\
 &&&&&&&&&\\
 1+\frac{1}{t_{g-1}} & (1+\frac{1}{t_1})(1+\frac{1}{t_{g-1}}) 	& (1+\frac{1}{t_2})(1+\frac{1}{t_{g-1}}) & \cdots & 1+\frac{1}{t_{g-1}}+\frac{1}{t_{g-1}^{2}} 	& 2 					& \cdots & 0 		& 0 	& 0 \\
 &&&&&&&&&\\
 	\vdots 		& \vdots 							&	\vdots 										& \udots & \vdots 										& \vdots  				& \ddots & \vdots	&\vdots &\vdots \\
 &&&&&&&&&\\
 1+\frac{1}{t_2} 	& (1+\frac{1}{t_1})(1+\frac{1}{t_2}) & 1+\frac{1}{t_2}+\frac{1}{t_2^2} 				& \cdots & 1+\frac{1}{t_2t_{g-1}} 						& 1+\frac{t_{g-1}}{t_2} & \cdots & 2 		& 0 	& 0 \\
  &&&&&&&&&\\
 1+\frac{1}{t_1} 	& 1+\frac{1}{t_1}+\frac{1}{t_1^2} 	& 1+\frac{1}{t_1t_2} 							& \cdots & 1+\frac{1}{t_1t_{g-1}} 						& 1+\frac{t_{g-1}}{t_1} & \cdots & 1+\frac{t_2}{t_1} & 2 & 0 \\
  &&&&&&&&&\\
 1 				& \frac{1}{t_1} 					& \frac{1}{t_2} 								& \cdots & \frac{1}{t_{g-1}} 							& t_{g-1} 				& \cdots & t_2 		& t_1 	& 1 \\
\end{array}
\right)\end{equation*}}
The characteristic polynomial of $A_n(\mathbf{t})$ is
\begin{equation}
\label{eqp0}
\det(uI_{2g}-A_{2g}(\mathbf{t})) = (u-1)^{2g-2}\left(u^2-\left(\sum_{i=1}^{g-1}t_i+2g+1+\sum_{i=1}^{g-1}\frac{1}{t_i}\right)u+1\right) = (u-1)\Theta_{g,0}
\end{equation}
For completeness, the row and column operations we perform in order to deduce the determinant formula above are the following. Let $t_0:=1$, $R_i$ denote the $i^{th}$ row and $C_i$ denote the $i^{th}$ column of  $A_{2g}-uI_{2g}$.

{
\begin{alignat*}{3}
&(1)\,\,R_i\to R_i-t_{i-1}R_{n+1-i},\,\, (\forall\, 1\leq i\leq g);  && (2)\,\,C_i\to C_i-\frac{1}{u}C_n,\,\, (\forall\, 2\leq i\leq n-1);\\
&(3)\,\,C_n\to C_n-\frac{u}{u-1}C_1;  && (4)\,\,C_i\to C_i+t_{n-i}C_{n+1-i}, (\forall\, g+1\leq i\leq n-1);\\
&(5)\,\,R_i\to R_i/\left(1+\frac{1}{t_{n-i}}\right),\, (\forall\, g+1\leq i\leq n-1);\,\quad\, && (6)\,\,C_i\to C_i/(1+t_{n-i}),\,\,( \forall\, g+1\leq i\leq n-1);\\
&(7)\,\,R_n\to R_n/(2-1/u). &&
\end{alignat*}
}

These operations reduce $A_{2g}-uI_{2g}$ to the form {\tiny$\begin{pmatrix}
X & \mathbf{0}\\
Y & Z
\end{pmatrix}$} where $X=(1-u)I_g$, and $Z$ is the $g\times g$ matrix with all the non-diagonal entries equal to $1$ and diagonal entries {\[\left(1+\frac{(1-u)t_{g-1}}{(1+t_{g-1})^2},\,\ldots,\, 1+\frac{(1-u)t_{1}}{(1+t_{1})^2},\, 1+\frac{(1-u)u}{2u-1}\right).\]} 

One can show using induction that a matrix with diagonal entries $(1+l_1, 1+l_2, ...,1+l_m)$ and all other entries equal to $1$ has determinant $l_1l_2... l_m + \sum_{j=1}^ml_1...\widehat{l_j}...l_m$, where $\widehat{l_j}$ indicates that $l_j$ is omitted from that product. In particular, when none of the $l_j$ are $0$, the determinant is \[l_1\cdots l_m\left(1+\frac{1}{l_1}+\cdots+\frac{1}{l_m}\right).\]

Using this, and multiplying by the factors carried over from operations $5, 6, 7$ above, we obtain the desired formula (\ref{eqp0}). Finally, by Theorem \ref{mainThm}, the Teichm\"uller polynomial of the associated fibered face is the polynomial above, with one less factor of $(u-1)$, as we wanted to show.

\subsection{The Polynomials $\Theta_{g,p}$} The Teichm\"uller polynomials $\Theta_{g,p}$ given in (\ref{polys}) can be obtained by slightly modifying the OBP $(\sigma,\mathbf{k}) = ((n,...,2,1),(2n-1,...,2n-1,n))$ that we used to obtain $\Theta_{g,0}$. We will use the following.

\begin{proposition}\label{prop_5}
Suppose ${(\sigma,\mathbf{k})}$ defines an admissible OBP of size $n$ and $\sigma(n)=1$. For any $p\in\mathbb{Z}_{\geq0}$, let $\mathbf{k'}\in\mathbb{N}^n$ be obtained from $\mathbf{k}$ by changing $k_n$ to $k_n+p(k_1+...+k_{n-1})$. Then ${(\sigma,\mathbf{k'})}$ also defines an admissible OBP. Moreover, under this change, the incidence matrix simply changes by adding $p$ copies of the first $(n-1)$ rows to the $n^{th}$ row.
\end{proposition}

\begin{proof}
Since we will use it a lot during the proof, let us denote by $r$ the sum of the first $(n-1)$ integers in $\mathbf{k}=(k_1,...,k_n)$.
\[r=k_1+...+k_{n-1}\]

Denote the blocks for the pair $(\sigma,\mathbf{k})$ as $B_1=\{1,...,k_1\}, ...,B_n=\{r+1,...,r+k_n\}$. As in $\S$\ref{obps}, $K=k_1+...+k_n=r+k_n$, so $\mathbb{N}_K = \coprod_i B_i$ and the block function $\beta:\mathbb{N}_K\to\mathbb{N}_n$ determines which block each element of $\mathbb{N}_K$ belongs to. Let $\xi$ be the associated admissible OBP defined by the equation (\ref{OBPEquation}) recalled here, \[\xi(j)\,\, := \sum_{1\le i<\sigma(\beta(j))}k_{\sigma^{-1}(i)} +j - \sum_{1\le i<\beta(j)}k_i.\]

The OBP simply permutes the blocks $B_i$ by permuting their indices using $\sigma$. For instance, as $\sigma(n)=1$, the last block $B_n$ is placed first, so that $\xi(r+1)=1$. For each $1\leq i\leq n$, let $O_i=(i,\xi(i),...,\xi^{(m_i-1)}(i))$ denote the $\xi-$orbit of $i$, until first return to $\mathbb{N}_n$. 

Similarly, let $\xi_2$ denote the OBP calculated for the pair $(\sigma,\mathbf{k'})$ with blocks $B_1',...,B_n'$, and let the $\xi_2-$orbits be $O_1',...,O_n'$.  As $k_i=k_i'$ for each $1\leq i\leq n-1$, we have $B_i=B_i'$ for every $i<n$. 

Now, $k_n'=k_n+pr$, so $B_n'=\{r+1,...,K':=(p+1)r+k_n\}$, large enough to contain $p$ copies of $B_1,...,B_{n-1}$ in addition to $B_n$. For each element $j\in B_n'$ we have $\beta(j)=n$, hence $\xi_2(j)=j-r$. On the other hand, for $j$ in $B_1',...,B_{n-1}'$, since $\beta(j)<n$, and thus $\sigma(\beta(j))>1$, the only change to the formula (\ref{OBPEquation}) above from $(\sigma,\mathbf{k})$ to $(\sigma,\mathbf{k'})$ is that $k'_{\sigma^{-1}(1)}=k'_n=k_n+pr$.  Thus the permutation $\xi_2:\mathbb{N}_{K'}\to\mathbb{N}_{K'}$ has the form:

{\[ \xi_2(j) = \begin{cases} 
      \xi(j)+pr \quad & 1\leq j\leq r \\
      j-r & r+1\leq j\leq K'
   \end{cases}
\]}

The $\xi_2-$orbits are thus easy to derive from the $\xi-$orbits. We may assume $p>0$. For every $i\leq n$, $O_i'$ starts with $i$ which is less than $r$, so the second element is $\xi_2(i)=\xi(i)+pr$, which is necessarily $>r$. The third element $\xi_2(\xi_2(i)) = \xi(i)+(p-1)r$ which is still $>r$ if $p>1$. So, the elements are reduced by $r$, $p$ times, until $\xi_2^{{p+1}}(i)=\xi(i)$. Now if $\xi(i)\leq r$ (i.e. $\xi(i)\notin B_n$), then the next element is $\xi(\xi(i))+pr$ and we continue as before, reducing by $r$. However, if $\xi(i)>r$ (i.e. $\xi(i)\in B_n$), then the next element is $\xi_2(\xi(i))=\xi(i)-r=\xi(\xi(i))$. 

In other words, between each pair of adjacent entries $(...,j,\xi(j),...)$ of $O_i$ such that $j\notin B_n$, we can insert $p$ elements $(...,j,\textcolor{Red}{\xi(j)+pr,...,\xi(j)+2r,\xi(j)+r},\xi(j),...)$ to obtain the orbit $O_i'$. Thus, for each $i\leq n$, we have  $O_i\subset O_i'$ and the difference $O_i'\setminus O_i \subset B_n'$.   So the incidence matrix $A'$ of ${(\sigma,\mathbf{k'})}$ can be obtained from the incidence matrix $A$ of ${(\sigma,\mathbf{k})}$ by adding $p$ copies of the first $(n-1)$ rows of $A$ to the last row of $A$. 

It is immediate to check the conditions of admissibility (Def.\ref{admissibledef}) for the OBP ${(\sigma,\mathbf{k'})}$. For condition (I), note that $\forall i\leq n$, the last element of $O_i$ and $O_i'$ are the same since every orbit of an admissible OBP must end in $B_{\sigma^{-1}(1)}$, which is $B_n$ by assumption. Hence, the first return to $\mathbb{N}_n$ is also the same for both; $\xi'=\xi_2'=\sigma$ which proves (I). 
For (II), note that \[\sum_{i=1}^n|O_i'| = \sum_{i=1}^n|O_i|+pr = K+pr=K',\] since we have added $p$ elements for each element of each $O_i$ which was in the first $(n-1)$ blocks, and the first $(n-1)$ blocks have $r$ elements altogether. Since distinct orbits $O_i'$ are disjoint, we have $\coprod O_i'=\mathbb{N}_{K'}$.
(III) follows from $O_i\subset O_i'$ and because $\min$ and $\max$ elements of $B_i$ and $B_i'$ are the same, except that $\max\{B_n'\}=K'=K+pr=\max\{B_n\}+pr$. By our assumption $K\in O_{\sigma^{-1}(n)}$. But $K=\xi_2(K+r)=...=\xi_2^{p}(K+pr)$, so $K'\in O_{\sigma^{-1}(n)}'$. For (IV), the matrix $A'$ is irreducible since its (non-negative) entries are not smaller than the corresponding entries of the irreducible matrix $A$.

\end{proof}

By Lemma \ref{skadmis} and Proposition \ref{prop_5}, we have $\forall g\geq2,\,p\geq0$, a pseudo-Anosov map $f_{g,p}$ on a surface of genus $g$, determined by the OBPs of size $n=2g$ given by 
\[\sigma=(n,...,2,1), \quad \mathbf{k}=\left(2n-1,...,2n-1,n+p(n-1)(2n-1)\right)\]

By the proof above, the difference in orbits $f_{g,p}(\gamma_i)$ and $f_{g,0}(\gamma_i)$ (\ref{p0orbits}) is that after each $\gamma_j\in f_{g,0}(\gamma_i)$ with $j<n$, we add $p$ copies of $\gamma_n$. In the basis $\{\gamma_1,...,\gamma_n\}$, the $f_{g,p}-$invariant cohomology still has the same basis $\{a^1,...,a^{g-1}\}$ given by (\ref{invariantCoh}), and the first Betti number of the mapping torus of $f_{g,p}$ is also still $g$. 

We will again denote the dual basis to the $\{a^i\}$ by $\{t_1,...,t_{g-1}\}$. Since $\gamma_n$ lifts to a closed loop in $\widetilde{S}$, adding copies of $\gamma_n$ doesn't change the coefficients of the elements in $\widetilde{f_{g,0}}(1\gamma_i)$. To obtain the lifted orbit $\widetilde{f_{g,p}}(1\gamma_i)$ we thus have to modify $\widetilde{f_{g,0}}(1\gamma_i)$ as follows: 

\begin{itemize}
\item For every $1 < i \leq g$, after each element $\mathbf{t^c}\gamma_i$, add in $p$ copies of $(t_{i-1})^{-1}\mathbf{t^c}\gamma_n$.

\item For every $g < i < n$, after each element $\mathbf{t^c}\gamma_i$, add in $p$ copies of $(t_{n-i})\mathbf{t^c}\gamma_n$ 
\end{itemize}

As a result, the only difference between $A_{g,p}(\mathbf{t})$ and $A_{g,0}(\mathbf{t})=A_{n}(\mathbf{t})$ computed in $\S$\ref{secg0} is in the last row. If we denote by $\Delta_i$ the $i^{th}-$row of $A_{n}(\mathbf{t})$ then the last row of $A_{g,p}(\mathbf{t})$ is \[\Delta_n +p\Delta_1 + pt_1^{-1}\Delta_2+\cdots+pt_{g-1}^{-1}\Delta_g+pt_{g-1}\Delta_{g+1}+\cdots+pt_1\Delta_{n-1}\] Removing these multiples from the last row of $(A_{g,p}(\mathbf{t})-uI_{n})$, and factoring out $(1+pu)$ from the last row yields a matrix identical to $A_n(\mathbf{t})$ except for the $(n, n)^{th}$ entry, which is now $\frac{1-u}{1+pu}$ instead of $(1-u)$. Computing the determinant as before yields the characteristic polynomial $(u-1)\Theta_{g,p}$, establishing $\Theta_{g,p}$ (\ref{polys}) as the Teichm\"uller polynomial of the OBP $(\sigma,\mathbf{k})$ above. 

\subsection{The Fibered Cone of $\Theta_{g,p}$}
\begin{lemma}\label{lemmaCone}The fibered cone of the polynomial \begin{equation}\label{quadraticPart}
\Theta'=u^2-\left(\sum_{i=1}^{g-1}t_i+2g+p+1+\sum_{i=1}^{g-1}\frac{1}{t_i}\right)u+1
\end{equation} is the same as the fibered cone for $(u-1)^k\Theta'$ for any $k\in\mathbb{Z}_{\geq0}$ and is given by \[\mathcal{C}=\{(s_1,...,s_{g-1},y):  y>0 \text{ and each } |s_i|<y\}\subset H^1(M;\mathbb{R}).\]
\end{lemma}

\begin{proof}

Let $\mathcal{C}=\mathbb{R}_+\cdot F$ where $F$ is the fibered face of the unit ball of the Thurston norm. We will use McMullen's result \cite[Theorem 6.1]{Mc00}, that that there exists a face $D$ of the Teichm\"uller norm unit ball such that $\mathcal{C}=\mathbb{R}_+\cdot F=\mathbb{R}_+\cdot D$. 

The Teichm\"uller norm is defined as follows: First, $\Theta_F$ is written as an element of $\mathbb{Z}[G]$ as $\Theta_F = \sum a_g\cdot g$, where $G=H_1(M;\mathbb{Z})/\text{torsion}$. Next, the Newton polygon $N(\Theta_F)$ is constructed as the convex hull of the finite set of integral homology classes $g$ with coefficient $a_g\neq0$. The Teichm\"uller norm is then defined for $\phi\in H^1(M;\mathbb{R})$ as \[||\phi||_{\Theta_F} = \sup_{a_g\neq0\neq a_h}\phi(g-h).\] 
As \cite{Mc00} explains, it measures the length of the projection of the Newton polygon, $\phi(N(\Theta_F))\subset\mathbb{R}$. In our case, $G=H_1(M;\mathbb{Z})/$torsion $\cong\mathbb{Z}^g$ with basis $\{t_1,...,t_{g-1},u\}$. Let $\phi=(s_1,...,s_{{g-1}},y)$ denote the coordinates of an arbitrary point $\phi\in H^1(M;\mathbb{R})$ satisfying $\phi(t_i)=s_i$ and $\phi(u)=y$.

For $\Theta'$ (\ref{quadraticPart}), the Newton polygon is a diamond with the vertices \[\{\mathbf{0},(\pm1,0,...,0,1),...,(0,...,0,\pm1,1),(0,...,0,2)\}\subset G.\] For this polygon, the Teichm\"uller norm is
\[||\phi||_{\Theta'} = ||\left(s_1,...,s_{g-1},y\right)||_{\Theta'} = \max_{1\,\leq\,i,\,j\,\leq\,g-1}\{|2y|,|y\pm s_i|,|s_i\pm s_j|\}\]
If each $|s_i|\leq|y|$, we have $|y\pm s_i|\leq |y|+|s_i|\leq 2|y|$ and similarly, $|s_i\pm s_j|\leq |s_i|+|s_j|\leq 2|y|$, for each $i,j$. So, for points satisfying $|s_i|<y$, the $\max$ is achieved by $2y$. And if say $|s_1|>|y|$, the max is not achieved by $|2y|$. Thus, the face of unit ball of the Teichm\"uller norm that intersects the $y$-axis is the cube at height $1/2$ given by \[D=\left\lbrace\left(s_1,...,s_{g-1},\frac{1}{2}\right):|s_i|<\frac{1}{2}\right\rbrace\subset H^1(M;\mathbb{R})\]

Now, for $\Theta = (u-1)^k\Theta'$, the Newton polygon is simply stretched in the $u$-direction. Namely, we keep the original vertices of the Newton polygon of $\Theta'$, and add the same vertices again, incremented by $k$ in the last component. Consequently, for $\Theta_{g,p}$ which is of the form $(u-1)^k\Theta'$, with $k=2g-3$, the Teichm\"uller norm is 

{\small\[||(\mathbf{s},y)||_{\Theta_{g,p}} = \max_{1\,\leq\,i,\,j\,\leq\,g-1}\left\lbrace|(k+2)y|,|(k+1)y\pm s_i|,...,|y\pm s_i|,|ky\pm s_i\pm s_j|,...,|y\pm s_i\pm s_j|,|s_i\pm s_j|\right\rbrace\]}

Now, if each $|s_i|<\frac{1}{k+2}$ and $y=\frac{1}{k+2}$, we have by triangle inequality that each term in the $\max$ above is $\leq1$, while $|(k+2)y|=1$. Reasoning as above, the face of the unit norm ball intersecting the $y$-axis is \[D=\left\lbrace\left(s_1,...,s_{g-1},\frac{1}{k+2}\right):|s_i|<\frac{1}{k+2}\right\rbrace\subset H^1(M;\mathbb{R})\]

Thus multiplying by a factor of $(u-1)$ in our case doesn't change the open fibered cone $\mathcal{C}$ of the Thurston norm, which is thus given by  \[\mathcal{C}=\mathbb{R}_+\cdot D=\{(s_1,...,s_{g-1},y):  y>0 \text{ and each } |s_i|<y\}\subset H^1(M;\mathbb{R}).\]\end{proof}

\subsection{Proof of Proposition \ref{Prop1}}\label{proofProp1}
Having calculated the fibered cone above, we can evaluate the polynomials $\Theta_{g,p}$ at any integral point $\phi=(s_1,...,s_{g-1},y)$ such that $y\in\mathbb{N}$ and each $s_i\in\mathbb{N}_{y-1}$. This yields a polynomial whose largest root is the stretch-factor of the pseudo-Anosov map corresponding to the monodromy associated to the fibration given by the point $\phi\in H^1(M;\mathbb{Z})\cong H_2(M;\mathbb{Z})$. 

If $\phi$ is \emph{primitive}, namely, if it isn't a positive integer multiple of another integral class, then the surface in $H_2(M;\mathbb{Z})$ representing it must be connected: it cannot be $k$ disjoint copies of some other surface, since otherwise $\frac{1}{k}\phi$ would be integral, and it cannot be a disjoint union of surfaces of distinct genera, since otherwise the mapping torus would be disconnected. We are now ready to prove Proposition \ref{Prop1}.

\begin{proof}[Proof of Proposition \ref{Prop1}]
Let $m\geq2$ and $a_1,...,a_{m-1}$ be non-negative integers. Set $g=1+a_1+...+a_{m-1}$ and assume the following vector $v\in\mathbb{Z}^g$ is primitive:\[v=(\underbrace{1,...,1}_{a_{m-1}},\underbrace{2,...,2}_{a_{m-2}},\,...\,,\underbrace{m-1,...,m-1}_{a_1},m)\in\mathbb{Z}^g\]

For any $p\in\mathbb{Z}_{\geq0}$, we can view $v$ as an element of $H^1(M;\mathbb{Z})$, where $M$ is the mapping torus of $f_{g,p}$. We know $v$ belongs to fibered cone of $\Theta_{g,p}$ by Lemma \ref{lemmaCone}. Evaluating $\Theta_{g,p}$ at $x^v$ and dividing by the factor $(x^m-1)^{2g-3}$ corresponding to the factor $(u-1)^{2g-3}$ in $\Theta_{g,p}$ we obtain the following palindromic polynomial
\[x^{2m}-\left(\underbrace{x+\cdots+x}_{a_{m-1}}+\,\cdots+\underbrace{x^{m-1}+\cdots+x^{m-1}}_{a_{1}}+2g+p+1+\underbrace{\frac{1}{x^{m-1}}+\cdots+\frac{1}{x^{m-1}}}_{a_{1}}+\,\cdots+\underbrace{\frac{1}{x}+\cdots+\frac{1}{x}}_{a_{m-1}}\right)x^m+1\]
\[=x^{2m}-a_1x^{2m-1}-\cdots-a_{m-1}x^{m+1}-(2g+p+1)x^m-a_{m-1}x^{m-1}-\cdots-a_1x+1\]

By \cite[Theorem 5.1]{Mc00}, the stretch factor of the monodromy of the fibration corresponding to $v\in H^1(M;\mathbb{Z})$ is given by the largest root of this polynomial. And the fiber corresponding to $v$ in $H_2(M;\mathbb{Z})$ is connected since $v$ is primitive. Thus the largest root $\lambda$ of each such polynomial is the stretch-factor of the monodromy of a pseudo-Anosov map on a connected surface. In particular, each such number is a biPerron algebraic unit. 

Now, $p\in\mathbb{Z}_{\geq0}$ was arbitrary. Thus, if we denote the coefficient $-(2g+p+1)$ of $x^m$ by $-a_m$, we see that $a_m\geq 2g+1 = 2(1+a_1+...+a_{m-1})+1 = 3+2(a_1+...+a_{m-1})$. This completes the proof of Proposition \ref{Prop1}.  
\end{proof}

\subsection{Conclusion} 
Ordered Block Permutations are especially suited to computing the Teichm\"uller polynomials, as we have hopefully demonstrated in this paper. The methods presented here can be refined in a lot of different ways, but the most exciting of these is hinted at by Proposition \ref{prop_5}, which provides a glimpse into the richness of arithmetic structures among surface homeomorphisms. However, that is the topic of an upcoming paper.\\ 

\bigskip

{\bf \noindent Acknowledgements: } I would like to thank my undergraduate advisor Dennis Sullivan who got me interested in Thurston's work on surfaces and my doctoral advisor John Hubbard with whom we defined OBPs. \\

The examples were found using code written in C++, and the pictures were drawn using Inkscape and Mathematica (Licensed to the American University of Sharjah).\\


\bibliographystyle{plain}
\bibliography{TeichPoly}

\end{document}